\documentclass{amsart}                 

\usepackage{amssymb,amsmath, amsfonts}

\usepackage{hyperref}

\usepackage{latexsym}

\theoremstyle{plain}
\newtheorem{theorem}{Theorem}[section]
\newtheorem*{theorem*}{Theorem}
\newtheorem{lemma}[theorem]{Lemma}

\newtheorem{proposition}[theorem]{Proposition}
\theoremstyle{remark}
\newtheorem{definition}[theorem]{Definition}

\newtheorem{remark}[theorem]{Remark}
\numberwithin{equation}{section}

\newcount\quantno
\everydisplay{\quantno=0}\everycr{\quantno=0}
\newcommand\quant{\advance\quantno by1
                      \ifnum\quantno=1\qquad\else\quad\fi\forall }
\newcommand\itemno[1]{(\romannumeral #1)}

\renewcommand\Re{\operatorname{\mathrm{Re}}}

\newcommand\Var{\operatorname{\rm{Var}}}

\newcommand\rest[1]{\kern-.1em
          \lower.5ex\hbox{$\scriptstyle #1$}\kern.05em}

\newcommand\set[1]{{\left\{#1\right\}}}

\renewcommand\mod[1]{\left\vert{#1}\right\vert}
\newcommand\bigmod[1]{\bigl\vert{#1}\bigr|}

\newcommand\norm[2]{{\Vert{#1}\Vert_{#2}}}

\newcommand\bignorm[2]{{\bigl\Vert{#1}\bigr\Vert_{#2}}}

\newcommand\opnorm[2]{|\!|\!| {#1} |\!|\!|_{#2}}

\newcommand\smallfrac[2]{\mbox{\small$\displaystyle\frac{#1}{#2}$}}
\newcommand\wrt{\,\text{\rm d}}

\newcommand\1{{\bf 1}}

\newcommand\BC{\mathbb{C}}

\newcommand\BN{\mathbb{N}}
\newcommand\BR{\mathbb{R}} \newcommand\BRd{\mathbb{R}^d}

\newcommand\BZ{\mathbb{Z}}

\newcommand\cB{\mathcal{B}}   

\newcommand\cC{\mathcal{C}}

\newcommand\cL{\mathcal{L}}   

\newcommand\cM{\mathcal{M}}  \newcommand\fM{\mathfrak{M}}

\newcommand\cQ{\mathcal{Q}}

\newcommand\cT{\mathcal{T}}

  \newcommand\fZ{\mathfrak{Z}}

\newcommand\al{\alpha}
\newcommand\be{\beta}
\newcommand\ga{\gamma}    
\newcommand\de{\delta}
  \newcommand\vep{\varepsilon}

\newcommand\la{\lambda}   
\newcommand\om{\omega}      
\newcommand\si{\sigma}
\newcommand\te{\theta}

\newcommand\funnyk{k\hbox to 0pt{\hss\phantom{g}}}

\newcommand\lu[1]{L^1(#1)}
\newcommand\lp[1]{L^p(#1)}

\newcommand\ld[1]{L^2(#1)}
\newcommand\ldO[1]{L^2_0(#1)}

\newcommand\ly[1]{L^\infty(#1)}

\newcommand\hu[1]{H^1(#1)}

\newcommand\whH{\widehat{\phantom{G}}\hbox to 0pt{\hss $H$}}

\newcommand\emspace{\hbox to 6pt{\hss}}
\newcommand\ds{\displaystyle}

{\mathsurround=0pt}

\newcommand\rmi{\hbox{\rm (i)}}
\newcommand\rmii{\hbox{\rm (ii)}}
\newcommand\rmiii{\hbox{\rm (iii)}}
\newcommand\rmiv{\hbox{\rm (iv)}}
\newcommand\rmv{\hbox{\rm (v)}}

\newcommand\ioty{\int_0^{\infty}}

\newcommand\One{{\mathbf{1}}}

\newcommand\e{\mathrm{e}}

\newcommand\diam[1]{\mathrm{diam}#1}

\newcommand\EB{{\rm{(AM)}}}
\newcommand\PP{{\rm{(I)}}}

\newcommand\IcM{I_{M,B_0}^c}
\newcommand\cIP{{\rm (I$\!\phantom{.}^c_{\!B_0}$)}}

\newcommand\supp{\mathrm{supp}\,}

\begin{document}

\title[$H^1$, $BMO$ for locally doubling spaces]
{$H^1$ and $BMO$ for certain locally doubling \\  metric measure spaces of
finite measure}



\thanks{Work partially supported by the
Progetto Cofinanziato ``Analisi Armonica''.}

\author[A. Carbonaro, G. Mauceri and S. Meda]
{Andrea Carbonaro, Giancarlo Mauceri and Stefano Meda}

\address{A. carbonaro, G. Mauceri: Dipartimento di Matematica \\
Universit\`a di Genova \\ via Dodecaneso 35, 16146 Genova \\ Italia}
\address{S. Meda: Dipartimento di Matematica e Applicazioni
\\ Universit\`a di Milano-Bicocca\\
via R.~Cozzi 53\\ 20125 Milano\\ Italy}

\begin{abstract}
In a previous paper the authors developed a $H^1-BMO$ theory for 
unbounded metric measure spaces $(M,\rho,\mu)$ of infinite measure that 
are locally doubling and satisfy two geometric properties, called 
``approximate midpoint" property and ``isoperimetric" property. In this 
paper we develop a similar theory for spaces of finite measure. 
We prove that all the results that hold in the infinite measure case have their 
counterparts in the finite measure case.
Finally, we show that the theory applies to a class of unbounded, 
complete Riemannian manifolds of finite measure and to a 
class of metric measure spaces of the form 
$(\BR^d,\rho_\varphi, \mu_\varphi)$, where 
$\wrt\mu_\varphi=\e^{-\varphi}\wrt x$ and $\rho_\varphi$ is the 
Riemannian metric corresponding to the length element 
$\wrt s^2=(1+\mod{\nabla\varphi})^2(\wrt x_1^2+\cdots+\wrt x_d^2)$.
This generalizes previous work of the last 
two authors for the Gauss space.
\end{abstract}

\maketitle

\setcounter{section}{0}
\section{Introduction} \label{s:Introduction}

In \cite{CMM} the authors developed a $H^1-BMO$ theory on  unbounded 
metric measure spaces $(M,\rho,\mu)$ that are locally doubling and 
satisfy two additional  ``geometric" properties, called  \emph{approximate 
midpoint}   (AM)  property and \emph{isoperimetric} (I) property. Roughly 
speaking, a space satisfies (AM) if its points do not become  too sparse 
at infinity and satisfies (I) if a fixed ratio of the measure of any bounded set 
is concentrated near the boundary.\par
For each scale parameter $b$ in $\BR^+$, we defined the spaces $H^1_b(\mu)$ and 
$BMO_b(\mu)$ much as in the classical case of spaces of homogeneous 
type, in the sense of Coifman and Weiss  \cite{CW}, the only difference 
being that the balls involved have at most radius $b$. Then we showed 
that these spaces do not depend on the scale $b$, at least if $b$ is 
sufficiently large, and that all the classical results that hold on spaces of 
homogeneous type, such as a John-Nirenberg inequality, the $H^1(\mu)-
BMO(\mu)$ duality, complex interpolation, hold for these spaces. 
Moreover these spaces provide end-point estimates for some interesting 
singular integrals which arise in various settings. We also showed that the 
theory applies to noncompact complete Riemannian manifolds with Ricci 
curvature bounded from below and strictly positive spectrum, e.g. to 
noncompact Riemannian symmetric spaces.\par
In \cite{CMM} we focused on the case where $\mu(M) = \infty$. In this 
paper we tackle the case where $\mu(M) <\infty$. In this case we must  
modify slightly the isoperimetric property, by assuming  that, instead of of 
(I), $M$ satisfies the \emph{complementary isoperimetric} property {\cIP}.  Roughly speaking, $M$ satisfies property {\cIP} if  
there exists a ball $B_0$ such that
a fixed ratio of the measure of any open set contained
in $M\setminus \Bar B_0$
is concentrated near the boundary of the set.  \par
When $\mu(M)$ is finite, the definitions of  the atomic Hardy space $\hu
{\mu}$
and the space $BMO({\mu})$ of functions of bounded mean oscillation  
are quite similar to those of the corresponding spaces
in the infinite measure case considered in \cite{CMM}. 
\par
To be specific, for each $b$ in $\BR^+$ denote by $\cB_b$ the collection
of balls of radius at most $b$.  The constant $b$ may be thought of as
a ``scale parameter'',
and the balls in $\cB_b$ are called \emph{admissible balls
at the scale $b$}.
An atom $a$ is either the exceptional atom $1/\mu(M)$
or a function in $\lu{\mu}$ supported in 
a ball $B$ which satisfies an appropriate ``size''
and cancellation condition.  
Fix a sufficiently large  ``scale parameter'' $b$ in $\BR^+$
(how large depends on 
the constants that appear in the definition of the (AM)
property).
Then $\hu{\mu}$ is the space of
all functions in $\lu{\mu}$ that admit a decomposition
of the form $\sum_j \la_j \, a_j$, where the $a_j$'s
are atoms supported in balls in $\cB_b$ or the exceptional atom,
and the sequence of complex numbers $\{\la_j\}$ is summable.

A locally integrable function $f$ is in $BMO(\mu)$ if 
it is in $\lu{\mu}$ and
$$
\sup_B \smallfrac{1}{\mu(B)} \int_B \mod{f-f_B} \wrt \mu < \infty,
$$
where the supremum is taken over \emph{all} balls $B$
in $\cB_b$, and $f_B$ denotes the average of $f$ over~$B$.
This definition of $BMO(\mu)$ is inspired by previous work of A.~Ionescu \cite{I}, who defined a similar space on rank one noncompact symmetic spaces.\par
We prove that these spaces indeed do not depend on the parameter $b$, 
that the topological dual of $H^1({\mu})$ is isomorphic to $BMO({\mu})$ 
and
an inequality of John--Nirenberg type holds for functions in $BMO(\mu)$. 
Furthermore, the spaces $\lp{\mu}$ are intermediate
spaces between $H^1(\mu)$ and $BMO(\mu)$ for the
complex interpolation methods.
It is worth observing that some important operators, which are bounded
on $\lp{\mu}$ for all $p$ in $(1,\infty)$, but otherwise unbounded
on $\lu{\mu}$ and on $\ly{\mu}$, turn out to be bounded
from $\hu{\mu}$ to $\lu{\mu}$ and from $\ly{\mu}$ to $BMO({\mu})$.

Some of the proofs  of these results require only simple adaptations
of the proofs of the analogous results in \cite{CMM}. In these cases we  
shall briefly 
indicate the variations needed. Other proofs,
like those of the duality and the interpolation results, require
more substantial changes, and we give full details.

In Section~\ref{s: Cheeger} we show that our theory applies to      
unbounded complete Riemannian manifolds $M$ of finite volume with 
Ricci curvature bounded from below 
such that Cheeger's isoperimetric constant 
$h(M)$ is strictly positive. It is well known that, on such manifolds, 
Cheeger's constant is strictly positive if and only if the Laplace--Beltrami
operator  $\cL$ on $M$ has spectral gap, i.e. if and only if $0$ is an 
isolated eigenvalue of $\cL$ on $L^2(\mu)$.
\par
In \cite{MM} G. Mauceri and S. Meda defined an atomic Hardy space
$\hu{\ga}$ and a space $BMO(\ga)$ of functions
of bounded mean oscillation associated to the Gauss measure $\wrt\ga(x)=\e^{-\mod{x}^2}\wrt x$ on $
\BR^d$. 
We recall briefly the definitions of these spaces.
For each scale parameter $b$ we denote
by $\cB_b^\ga$ the set of all Euclidean balls $B$ in $\BR^d$
such that 
$$
r_B \leq b \, \min\bigl(1,1/\mod{c_B}\bigr),
$$
where $c_B$ and $r_B$ denote the centre and the radius of $B$
respectively.  Now, $\hu{\ga}$ is defined as $\hu{\mu}$
above, but with the family of admissible balls $\cB_b$
replaced by $\cB_b^\ga$, and similarly for $BMO(\ga)$.
In \cite{MM} the authors proved that $\hu{\ga}$  
and $BMO(\ga)$ possess the analogues of the properties
enumerated above for $\hu{\mu}$ and $BMO(\mu)$.
They also showed that some important operators related to the 
Ornstein--Uhlenbeck operator on $\BR^d$ that are bounded
on $\lp{\ga}$ for all $p$ in $(1,\infty)$, but otherwise unbounded
on $\lu{\ga}$ and on $\ly{\ga}$, are be bounded
from $\hu{\ga}$ to $\lu{\ga}$ and from $\ly{\ga}$ to $BMO({\ga})$.
\par
It may be worth observing that the measured metric space $(\BR^d,
\rho, \ga)$, where $\rho$ denotes the Euclidean distance, has finite 
measure and is not locally doubling.
\par
The definition of the class $\cB_b^\gamma$ of admissible balls in \cite
{MM} suggests that on the Gauss space $(\BR^d,\rho,\gamma)$ the 
Euclidean metric $\rho$ should be replaced by the Riemannian metric 
associated to the length element $\wrt s^2=(1+\mod{x})^2(\wrt x_1^2+
\cdots+\wrt x^2_d)$. 

\par
In Section \ref{s: Another} we exploit and generalize this idea, by considering metric 
measure spaces of the form $(\BR^d,
\rho_\varphi,\mu_\varphi)$ where $\varphi$ is a function in
$C^2(\BR^d)$, $\rho_\varphi$ is the Riemannian metric on $\BR^d$ 
defined by
the length element $\wrt s^2=(1+\mod{\nabla \varphi})^2\,(\wrt
x_1^2+\cdots+\wrt x_d^2)$ and
$\wrt\mu_\varphi=\e^\varphi\wrt\lambda$, where $\lambda$ is the 
Lebesgue measure on $\BR^d$. We prove that, if the function $\varphi$ 
satisfies appropriate conditions, the space $(\BR^d,
\rho_\varphi,\mu_\varphi)$ is locally doubling and satisfies properties 
(AM) and~{\cIP}.\par
Finally, we recall that Hardy spaces and spaces of functions of bounded
mean oscillation have recently
been studied on various nondoubling metric measure spaces
\cite{MMNO, NTV, To, V}. We point out that our spaces are different and 
that they provide end-point estimates for singular integrals which do not 
satisfy the standard Calder\'on-Zygmund estimates at infinity, still 
mantaining the important property that the complex interpolation spaces 
between $H^1(\mu)$ and  $BMO(\mu)$ are the spaces $L^p(\mu)
$.

\section{Geometric assumptions} \label{s: PPc}

Suppose that $(M,\rho,\mu)$ is a metric measure space and denote
by $\cB$ the family of all balls in $M$. We assume that
$0<\mu(M)<\infty$. For each
$B$ in $\cB$ we denote by $c_B$ and $r_B$ the centre and the radius
of $B$ respectively. Furthermore, for each $\kappa>0,$ we denote by
$\kappa \, B$ the ball with centre $c_B$ and radius $\kappa \, r_B$.
For each $b$ in $\BR^+$, we denote by $\cB_b$ the family of all
balls $B$ in $\cB$ such that $r_B \leq b$.
For any subset $A$ of $M$ and each $\kappa$ in $\BR^+$
we denote by $A_{\kappa}$ and $A^{\kappa}$ the sets
$$
\bigl\{x\in A: \rho(x,A^c) \leq \kappa\bigr\}
\qquad\hbox{and}\qquad
\bigl\{x\in A: \rho(x,A^c) > \kappa\bigr\}
$$
respectively.

In this paper we assume that $(M,\rho,\mu)$ is an
unbounded measured metric space of \emph{finite measure},
which possesses the following properties:
\begin{enumerate}
\item[\itemno1]
\emph{local doubling property} (LD):
for every $b$ in $\BR^+$ there exists a constant $D_b$
such that
$$
\mu \bigl(2 B\bigr)
\leq D_b \, \mu  \bigl(B\bigr)
\quant B \in \cB_b;  
$$
This property is often called \emph{local doubling condition}
in the literature, and we adhere to this terminology.
Note that if (LD) holds and $M$ is bounded, then $\mu$ is doubling.
\item[\itemno2]
\emph{property} \hbox{\EB} (approximate midpoint property): there
exist $R_0$ in $[0,\infty)$ and $\be$ in $(1/2,1)$ such that for
every pair of points~$x$ and $y$ in $M$ with $\rho(x,y) > R_0$ there
exists a point $z$ in $M$ such that $\rho(x,z)<\beta\,\rho(x,y)$ and $\rho
(y,z)<\beta\,\rho(x,y)$. 
\item[\itemno3]
\emph{complementary isoperimetric property} \hbox{{\cIP}}:
there exist a ball $B_0$ in $M$, 
$\kappa_0$ and $C$ in $\BR^+$ such that
for every open set $A$ contained in $M\setminus \Bar B_0$
\begin{equation}\label{PIc}
\mu \bigl(A_{\kappa}\bigr)
\geq C\, \kappa \, \mu (A) 
\quant \kappa \in (0,\kappa_0]. 
\end{equation}
Suppose that $M$ has property~{\cIP}.  For each $t$ in $(0,\kappa_0]$
we denote by $C_{t}$
the supremum over all constants $C$ for which (\ref{PIc}) holds for all 
$\kappa$ in $(0,t]$.  Then we define ${\IcM}$ by
$$
{\IcM} = \sup \bigl\{ C_{t}: t\in (0,\kappa_0] \bigr\}.
$$
Note that the function $t\mapsto C_t$ is decreasing
on $(0,\kappa_0]$, so that
\begin{equation}\label{IcM}
{\IcM} = \lim_{t\to 0^+} C_{t}.
\end{equation}
\end{enumerate}
\begin{remark}\label{r: compare}
The first two geometric assumptions (LD) and (AM) coincide with the 
corresponding assumptions made in \cite{CMM} for  spaces of infinite 
measure. The isoperimetric property is sligthly different from the 
isoperimetric property (I) in \cite{CMM}, because in the infinite measure 
case  we assumed that inequality (\ref{PIc}) holds for all  bounded open 
set in $M$.
\end{remark}
\begin{remark} \label{r: geom I}
The local doubling property implies that for each $\tau \geq 2$
and for each $b$ in $\BR^+$
there exists a constant~$C$ such that
\begin{equation} \label{f: doubling D}
\mu\bigl(B'\bigr)
\leq C \, \mu(B)
\end{equation}
for each pair of balls $B$ and $B'$, with $B\subset B'$,
$B$ in $\cB_b$, and $r_{B'}\leq \tau \, r_B$.
We shall denote by $D_{\tau,b}$ the smallest constant for
which (\ref{f: doubling D}) holds.
In particular, if (\ref{f: doubling D}) holds (with the same constant)
for all balls $B$ in $\cB$, then $\mu$ is doubling and
we shall denote by $D_{\tau,\infty}$ the smallest constant for
which (\ref{f: doubling D}) holds.
\end{remark}
\begin{remark}\label{r: AMP}
Loosely speaking, the approximate midpoint property means that the 
points of $M$ ``do not become to sparse at infinity". The properties is 
obviously satisfied on all length metric spaces.

\end{remark}
\begin{remark}\label{r: Ic}
In Section \ref{s: Cheeger} we shall see that, on complete Riemannian manifolds,  the 
complementary isoperimetric property is  equivalent to the positivity of 
Cheeger's isoperimetric costant
$$
h(M)= \inf  \frac{\sigma(\partial(A)}{\mu(A)}
$$
where the infimum runs over all bounded open sets $A$ with $\mu(A)\le 
\mu(M)/2$ and with smooth boundary  $\partial(A)$. Here $\sigma$ 
denotes the induced Riemannian measure on $\partial A$. Moreover, if the 
Ricci curvature of $M$ is bounded from below, both properties are 
equivalent to the existence of a spectral gap for the Laplacian.
\end{remark}
\begin{remark}
The local doubling property is needed for all the results in this
paper, but many results in Sections~\ref{s: PPc}-\ref{s:
interpolation} depend only on some but not all the properties
\rmi-\rmiii. In particular,  all the results in
Sections~\ref{s: H1 and BMO} and \ref{s: duality} require property
\EB\ but not property {\cIP}; 
Lemma~\ref{l: rdi II} and Theorem~\ref{t: basic II}, which are key
in proving the interpolation result Theorem~\ref{t: interpolation},
require property {\cIP}, but not property \EB. Finally, all the
properties \rmi-\rmiii\ above are needed for the interpolation
results and the theory of singular integral operators in
Section~\ref{s: interpolation}.
\end{remark}
\begin{proposition} \label{p: PPc}
Suppose that $M$ possesses property {\cIP}. The following hold:
\begin{enumerate}
\item[\itemno1]
for every open set $A$ contained in $M\setminus \Bar{B}_0$
$$
\mu(A_t) \geq \bigl(1-\e^{-{\IcM}t }\bigr) \, \mu(A) \quant t \in
\BR^+;
$$
\item[\itemno2]
For every point $x$ in $M$ there exists a constant $C$, which depends on 
$x$, $\IcM$ and $B_0$, such that
$$
\mu\big(B(x,r)^c\big)\le C\,\e^{-{\IcM}r} \qquad\forall r>0.
$$ 
\end{enumerate}
\end{proposition}

\begin{proof}
The proof of \rmi\ is almost \emph{verbatim} the same as the proof
of \cite[Proposition~3.1]{CMM}, and is omitted.

Now we prove \rmii. Denote by $V_r$ the measure of $B(x,r)^c$. Since $
\mu\big(B(x,r)^c\big)\le \mu(M)$ for every $r>0$, it is clearly enough to 
prove the inequality for $r$ sufficiently large, say $r>r_{B_0}+d(x,c_{B_0})
+1$. Then $B(x,r-1)^c\subset M\setminus \Bar{B}_0$ and $B(x,r-1)^c
\setminus B(x,r)^c\supseteq \big(B(x,r-1)^c\big)_1$. Thus, by \rmi
\begin{align*}
V_{r-1}-V_r\ge&\mu\left(\big(B(x,r-1)^c)_1\right) \\ 
\ge&(1-\e^{-{\IcM}})\, V_{r-1}. 
\end{align*}
Hence
$$
V_r\le \e^{-{\IcM}}\  V_{r-1}.
$$
By iteration, if $r_{B_0}+d(x,c_{B_0})+n<r\le r_{B_0}+d(x,c_{B_0})+n+1$ 
we obtain that
$$
V_r\le \e^{-{\IcM}n}\  V_{r-n}\le C\ \e^{-{\IcM}r},
$$
where $C=\exp\big((r_{B_0}+d(x,c_{B_0})+1){\IcM}\big)\,\mu(M)$.
\end{proof}

\section{$H^1$ and $BMO$} \label{s: H1 and BMO}

In this section we define the Hardy space $H^1(\mu)$ and the space 
$BMO(\mu)$. The definitions are very similar to those given in \cite{CMM} 
for metric spaces of infinite measure. The only differences are the 
existence of the ``exceptional atom"  in $H^1(\mu)$ and the fact that 
$BMO(\mu)$ is defined as a subspace of $L^1(\mu)$.\par
\begin{definition} \label{d: standard atom}
Suppose that $r$ is in $(1,\infty]$.  A $(1,r)$-\emph{standard atom} $a$
is a function in $\lu{\mu}$ supported in a ball $B$ in $\cB$
with the following properties:
\begin{enumerate}
\item[\itemno1]
$\norm{a}{\infty}  \leq \mu (B)^{-1}$
if $r$ is equal to $\infty$ and
$$
\Bigl(\frac{1}{\mu (B)} \int_B \mod{a}^r \wrt{\mu } \Bigr)^{1/r}
\leq \mu (B)^{-1}
$$
if $r$ is in $(1,\infty)$;
\item[\itemno2]
$\ds \int_B a \wrt \mu  = 0$.
\end{enumerate}
The constant function $1/\mu(M)$ is referred to as the
\emph{exceptional atom}.
\end{definition}

\begin{definition} \label{d: Hardy}
Suppose that $b$ is in $\BR^+$ and that $r$ is in $(1,\infty]$.
The \emph{Hardy space} $H_b^{1,r}({\mu})$ is the
space of all functions~$g$ in $\lu{\mu}$
that admit a decomposition of the form
\begin{equation} \label{f: decomposition}
g = \sum_{k=1}^\infty \la_k \, a_k,
\end{equation}
where $a_k$ is either a $(1,r)$-atom \emph{supported in a ball $B$ of $
\cB_b$}
or the exceptional atom,
and $\sum_{k=1}^\infty \mod{\la_k} < \infty$.
The norm $\norm{g}{H_b^{1,r}({\mu})}$
of $g$ is the infimum of $\sum_{k=1}^\infty \mod{\la_k}$
over all decompositions (\ref{f: decomposition})
of $g$.
\end{definition}

\begin{definition} \label{d: BMO II}
Suppose that $b$ is in $\BR^+$ and that $q$ is in $[1,\infty)$.
For each locally integrable function~$f$ we define $N_b^q(f)$ by
$$
N_b^q(f)
=  \sup_{B\in\cB_b} \Bigl(\frac{1}{\mu(B)}
\int_B \mod{f-f_B}^q \wrt\mu \Bigr)^{1/q},
$$
where $f_B$ denotes the average of $f$ over $B$. We denote by
$BMO_b^q(\mu)$ the space of all  functions~$f$ in $L^1(\mu)$ such that
$N^q_b(f)$ is finite, endowed with the norm
$$
\norm{f}{BMO_b^q(\mu)}
= \norm{f}{1} + N_b^q(f).
$$
\end{definition}

Note that only balls of radius at most $b$ enter in the definitions
of $H_b^{1,r}(\mu)$ and $BMO_b^q(\mu)$.

It is a nontrivial fact that $H_b^{1,r}(\mu)$ and~$BMO_b^q(\mu)$ are
independent of the parameter $b$, provided $b$ is large enough.
Recall that $R_0$ and $\be$ are the constants which appear in the
definition of the \EB\ property.

\begin{proposition}\label{p: decphi}
Suppose that 
$r$ is in $(1,\infty]$, $q$ is in $[1,\infty)$, and
$b$ and $c$ are in $\BR^+$ and satisfy $R_0/(1-\beta)<c<b$.
The following hold:
\begin{enumerate}
\item[\itemno1]
the identity is a Banach space isomorphism between
$H_{c}^{1,r}(\mu)$ and $H_{b}^{1,r}(\mu)$ and
between
$BMO_b^q(\mu)$ and $BMO_c^q(\mu)$;
\item[\itemno2]
(John-Nirenberg type inequality)
there exist positive constants $c$ and $C$ such that
for all $f \in BMO_b^1(\mu)$ and all $B$ in $\cB_b$
$$
\mu\bigl(\{x\in B: \mod{f(x)-f_B} >s \} \bigr)
\leq C \, \e^{-c \,s/N_b^1(f) } \, \mu(B)
;
$$
\item[\itemno3]
for each $q$ in $(1,\infty)$ there exists a constant $C$ such that
$$
N_b^1(f) \leq  N_b^q(f) \leq C\, N_b^1(f) \quant f \in BMO^q_b(\mu).
$$
\end{enumerate}
\end{proposition}

\begin{proof}
The proof of \rmi\ is \emph{almost verbatim} the same
as the proofs of \cite[Prop.~4.3]{CMM} and \cite[Prop.~5.1]{CMM}
respectively, and is omitted. The proof of \rmii\ is the same as
the proof of \cite[Thm~5.4]{CMM}, and the proof of \rmiii\ follows
the lines of the proof of \cite[Corollary~5.5]{CMM}.
\end{proof}

Suppose that $b$ and $c$ are in $\BR^+$ and satisfy $R_0/(1-\beta)<c<b
$.
In view of Proposition~\ref{p: decphi}~\rmii-\rmiv, if $q$ and $r$
are in $[1,\infty)$, then the identity is a Banach space isomorphism
between $BMO_b^q(\mu)$ and $BMO_c^r(\mu)$. We denote simply by
$BMO(\mu)$ the Banach space $BMO_b^q(\mu)$ endowed with any of 
the
equivalent norms~$N_b^q$.

Similarly, in view of Proposition~\ref{p: decphi}~\rmi, if $r$ is in
$(1,\infty)$, and then $H_{b}^{1,r}(\mu)$ and $R_0/(1-\beta)<c<b$, then
$H_{c}^{1,r}(\mu)$ are isomorphic Banach spaces, and they will
simply be denoted by $H^{1,r}(\mu)$. In Section~\ref{s: duality} we
shall prove that the topological dual of $H^{1,r}({\mu})$ may be
identified with $BMO^{r'}({\mu})$, where $r'$ denotes the index
conjugate to $r$. Suppose that $1<r<s<\infty$.  Then
$\bigl(H^{1,r}({\mu})\bigr)^* = \bigl(H^{1,s}({\mu})\bigr)^*$,
because we have proved that $BMO^{r'}({\mu})= BMO^{s'}({\mu})$.
Observe that the identity is a continuous injection of
$H^{1,s}(\mu)$ into $H^{1,r}(\mu)$, and that $H^{1,s}(\mu)$ is a
dense subspace of $H^{1,r}(\mu)$.  Then we may conclude that
$H^{1,s}(\mu) = H^{1,r}(\mu)$. Then we shall denote $H^{1,r}(\mu)$
simply by $\hu{\mu}$.

\section{Duality} \label{s: duality}

In this section, we prove the analogue of the duality result
\cite[Thm~6.1]{CMM}. The proof in the finite measure case  is more 
difficult because we must show that for every linear functional $\ell$ in the 
dual of $H^1(\mu)$ the function $f^\ell$ that represents  the functional on 
$H^1(\mu)\cap L^2_c(\mu)$, constructed in \cite[Thm~6.1]{CMM}, is also 
in $L^1(\mu)$.

We need more notation and some preliminary observation.
Suppose that $b>R_0/(1-\beta)$, where $R_0$ and $\beta$ are the 
constants in the approximate midpoint property (AM) (see Section \ref{s: 
PPc}).
A ball $B$ in $\cB_b$ is said to be \emph{maximal} if
$
r_B = b.
$

We shall make use of the analogues in our setting
of the so-called dyadic cubes $Q_\al^k$
introduced by G.~David and M.~Christ \cite{Da,Ch} on spaces of
homogeneous type. 

\begin{theorem} \label{t: dyadic cubes}
There exists a collection of open subsets $\{Q_\al^k: k \in \BZ, \al \in I_k\}$
and constants $\de$ in $(0,1)$, $a_0$, $C_1$ in $\BR^+$ such that
\begin{enumerate}
\item[\itemno1]
$\bigcup_{\al} Q_\al^k$ is a set of full measure in $M$
for each $k$ in $\BZ$;
\item[\itemno2]
if $\ell \geq k$, then either $Q_\be^\ell \subset Q_\al^k$ or
$Q_\be^\ell \cap Q_\al^k = \emptyset$;
\item[\itemno3]
for each $(k,\al)$ and each $\ell < k$ there is a unique $\be$
such that $Q_\al^k \subset Q_\be^\ell$;
\item[\itemno4]
$\diam (Q_\al^k) \leq C_1 \, \de^k$;
\item[\itemno5]
each $Q_\al^k$ contains some ball $B(z_\al^k, a_0\, \de^k)$.
\end{enumerate}
\end{theorem}
 It may help to think of $Q_\al^k$
as being essentially a cube of diameter $\de^k$ with ``centre" $z_\al^k$.
Note that \rmiv\ and \rmv\ imply that for every integer $k$
and each $\al$ in $I_k$
$$
B(z_{\al}^k, a_0 \, \de^k)
\subset Q_{\al}^k
\subset B(z_{\al}^k, C_1 \, \de^k/2).
$$
\begin{remark}
When we use dyadic cubes, we implicitly assume that for each
$k$ in $\BZ$ the set $M \setminus \bigcup_{\al \in I_k} Q_{\al}^k$
has been permanently deleted from the space.
\end{remark}
We shall denote by $\cQ^k$ the class of
all dyadic cubes of ``resolution'' $k$, i.e., the family of
cubes $\{Q_{\al}^k: \al \in I_k\}$, and by $\cQ$
the set of all dyadic cubes. 
We denote by $\fZ^{\nu}$ the set $\{z_{\al}^\nu: \al \in I_{\nu}\}$, i.e. the 
set of ``centres"  of all dyadic cubes of ``resolution" $\nu$. We recall that, 
in Christ's construction of the family $\cQ$ of
dyadic cubes, the set $\fZ^\nu$ is a maximal
collection of points in $M$ such that
$$
\rho(z_{\al}^\nu,z_{\be}^\nu) \geq \de^{\nu}
$$
for all $\al$, $\be$ in $I_{\nu}$, with $\al \neq \be$.

We shall need the following additional properties of dyadic cubes.

\begin{lemma} \label{l: geometric}
Choose an integer $\nu$ such that $\de^{\nu}\, \min(1,2a_0) > R_0$
and $b$ in $\BR^+$ such that $b> 4\de^\nu\, \max(1/(1-\be),a_0)$.
For each $z^\nu_\al$ in $\fZ^{\nu}$ denote by $B_{\al}$ the ball
$B(z_{\al}^\nu, b)$. The following hold:
\begin{enumerate}
\item[\itemno1]
 the balls $\{B_{\al}\}$ form a
locally uniformly finite covering of $M$, i.e. there exists an
integer $N_0$ such that
$$
1 \leq \sum_{j\in \BN} \One_{B_{\al}} \leq N_0;
$$
\item[\itemno2]
for every pair $o$, $z$ of distinct points in $\fZ^\nu$,
there exists a chain of $N$
points $z_{\al_1}^\nu, \ldots, z_{\al_N}^\nu$ in $\fZ^\nu$ such that
$o=z_{\al_1}^\nu$, $z =  z_{\al_N}^\nu$,
$$
N\leq 4\, \Bigl(\frac{2d}{b}\Bigr)^{1/[1-\log_2(1+\be)]} + 1
\qquad\hbox{and}\qquad
\rho(z_{\al_{j}}^\nu,z_{\al_{j+1}}^\nu)
< b/2,
$$
where $d$ denotes the distance $\rho(o,z_{\al}^{\nu})$.
Furthermore, for $j\in\{1,\ldots, N-1\}$ the intersection
$B_{\al_j} \cap B_{\al_{j+1}}$ contains the ball
$B(z_{\al_{j+1}},a_0\, \de^\nu)$, and
\begin{equation} \label{f: ratio of balls}
\frac{\mu\bigl(B_{\al_{j+1}}\bigr)}{\mu\bigl(B(z_{\al_{j+1}}^\nu,
a_0\, \de^\nu)\bigr)} \leq D_{b/(a_0\de^\nu),a_0\, \de^\nu}.
\end{equation}
\end{enumerate}
\end{lemma}

\begin{proof}
First we prove \rmi. By the maximaliy of the collection $\fZ^\nu$,
for each $x$ in $M$ there exists $z^\nu_\al$ in $\fZ^{\nu}$
such that $\rho(z_{\al}^\nu,x)<\de^\nu$.  This implies 
the
left inequality in \rmi.

A simple variation of the proof of \cite[Prop. 3.4 \rmiv]{CMM} shows that 
there exists an integer $N_0$, which depends on $b$, $\nu$, $a_0$ and 
$C_1$, such that a ball of radius $2b$ intersects at most $N_0$ cubes in 
$\cQ^\nu$. Let $A(x)=\set{B_\alpha: x\in B_\alpha}$. Since $z^\nu_\alpha
\in B_\alpha$ and  
$\bigcup_{B_\alpha\in A(x)} B_\alpha \subset B(x, 2b)$, 
the cubes $Q^\nu_{z^\nu_\alpha}$, $B_\alpha\in A(x)$, intersect $B(x,2b)$. 
Thus the cardinality of $A(x)$ is at most $N_0$. This proves  the right 
inequality in \rmi.

Next we prove \rmii.  Recall that $d$ denotes the distance
between $o$ and $z$.  Denote by $B^o$ and $B^z$ the balls
with radius $b$ centred at $o$ and $z$ respectively.

First suppose that $d< b/2$.  Then the chain reduces to the two points $o$ 
and $z$. Moreover $B^o \cap B^z$ contains
the ball $B(z,a_0\, \de^\nu)$.
Indeed, $B^z$
contains $B(z,a_0\, \de^\nu)$
(recall that $b> 4a_0\, \de^\nu$),
and $B^o$ contains $B(z,a_0\, \de^\nu)$,
because $B^o$ has radius $b$ and $b>b/2+a_0 \, \de^\nu$
is equivalent to $b>2a_0 \, \de^{\nu}$, which we assume.

Next suppose that $d\geq b/2$.  Since $b/2 > R_0$, there exists
a point $z_1$ in $M$ such that
$$
\max \bigl(\rho(z_1,o),\rho(z_1,z)\bigr)
< \be \, d
$$
by the \EB\ property.  In general $z_1$ need not be in $\fZ^\nu$.
However, by the maximality of $\fZ^\nu$,  there exists $z^\nu_{\al_1}$ in $
\fZ^\nu$ such that
$\rho(z^\nu_{\al_1},z_1) < \de^\nu$. We observe that
$$
\max \bigl(\rho(z_{\al_1}^\nu,o),\rho(z_{\al_1}^\nu,z)\bigr)
< \frac{1+\be}{2} \, d.
$$
Indeed, by the triangle inequality
$$
\begin{aligned}
\rho(z_{\al_1}^\nu,o)
& \leq \rho(z_{\al_1}^\nu,z_1) + \rho(z_1,o) \\
& \leq \de^\nu + \be \, d.
\end{aligned}
$$
Now, note that the conditions $d\geq b/2$ and $b> 4\de^\nu/(1-\be)$
imply $\de^{\nu} < (1-\be)\, d/2$, and we may conclude that
$$
\rho(z_{\al_1}^\nu,o)
< \Bigl(\frac{1-\be}{2} + \be \Bigr) \, d = \frac{1+ \be}{2} \, d.
$$
Similarly, we may show that $\rho(z_{\al_1}^\nu,z)< (1+\be)d/2$.

We have now a {chain} consisting of three ordered points
$o$, $z_{\al_1}^\nu$ and $z$.  The distance of
two subsequent points is $< (1+\be) \, d/2$.

Now consider the first two points $o$ and $z_{\al_1}^\nu$ of the
chain. If their distance is $<b/2$, then $B^o\cap B_{\al_1}$
contains the ball $B(z_{\al_1}^\nu,a_0\, \de^\nu)$. If, instead,
their distance is $\geq b/2$, then we may repeat the argument above,
and find $z^\nu_{\al_{1}^{(2)}}$ in $\fZ^\nu$ such that
$$
\max \bigl(\rho(z_{\al_{1}^{(2)}}^\nu,o),
\rho(z_{\al_{1}^{(2)}}^\nu,z_{\al_1}^\nu)\bigr)
< \Bigl(\frac{1+\be}{2}\Bigr)^2 \, d.
$$
Next we consider the two points $z_{\al_1}^\nu$ and $z$ of
the chain and argue similarly. Either their distance is $<b/2$, and
$B^z\cap B_{\al_1}$ contains the ball $B(z,a_0\,
\de^\nu)$, or their distance is $\geq b/2$, and we may find
$z^\nu_{\al_{2}^{(2)}}$ in $\fZ^\nu$ such that
$$
\max \bigl(\rho(z_{\al_{2}^{(2)}}^\nu,z_{\al_1}^\nu),
\rho(z_{\al_{2}^{(2)}}^\nu,z)\bigr)
< \Bigl(\frac{1+\be}{2}\Bigr)^2 \, d.
$$

By iterating the procedure described above $n$ times, we find a chain of
points $z_{\al_1}^\nu, \ldots, z_{\al_N}^\nu$, such that
$o=z_{\al_1}^\nu$, $z =  z_{\al_N}^\nu$, such that
$$
\rho(z_{\al_{j}}^\nu,z_{\al_{j+1}}^\nu)
< \Bigl(\frac{1+\be}{2}\Bigr)^n \, d
\quant j \in \{1,\ldots, N-1\}.
$$
If $n$ is the least integer $\geq \log_2(2d/b)/\log_2[2/(1+\be)]$, then
$$
\Bigl(\frac{1+\be}{2}\Bigr)^n \, d
< b/2,
$$
and for all $j$ in $\{1,\ldots, N-1\}$ the intersection
$B_{\al_j} \cap B_{\al_{j-1}}$ contains the ball
$B(z_{\al_{j+1}},a_0\, \de^\nu)$.  Furthermore, the number $N$
of points of the chain is at most
$$
4\, \Bigl(\frac{2d}{b}\Bigr)^{1/[1-\log_2(1+\be)]} + 1,
$$
and
$$
\frac{\mu\bigl(B_{\al_{j+1}}\bigr)}{\mu\bigl(B(z_{\al_{j+1}}^\nu,
a_0\, \de^\nu)\bigr)} \leq D_{b/(a_0\de^\nu),a_0\, \de^\nu}
$$
for all $j$ in $\{1,\ldots, N-1\}$, by the locally doubling property.

This concludes the proof of \rmii.
\end{proof}

We need more notation and some preliminary observations. Let
$b>0$. For each ball $B$ in $\cB_b$ let $\ldO{B}$ denote the Hilbert
space of all functions $f$ in $\ld{\mu}$ such that the support of
$f$ is contained in $B$ and $\int_B f \wrt {\mu } = 0$. We remark
that a function $f$ in $\ldO{B}$ is a multiple of a $(1,2)$-atom,
and that, for all $c\geq b$,
\begin{equation}  \label{f: basic estimate atom}
\norm{f}{H^{1,2}_c({\mu})} \leq \mu (B)^{1/2} \, \norm{f}{\ld{B}}.
\end{equation}

Let $\ell$ be a bounded linear functional on $H^{1,2}({\mu})$. Then,
for each $B$ in $\cB$ the restriction of $\ell$ to $\ldO{B}$ is a
bounded linear functional on $\ldO{B}$.  Therefore, by the Riesz
representation theorem there exists a unique function $\ell^B$ in
$\ldO{B}$ which represents the restriction of $\ell$ to $\ldO{B}$.
Note that for every constant $\eta$ the function $\ell^B+\eta$
represents the same functional, though it is not in $\ldO{B}$ unless
$\eta$ is equal to $0$. Denote by $\norm{\ell}{H^{1,2}(\mu)^*}$
the norm of $\ell$. Then, by (\ref{f: basic estimate atom}), we have
\begin{equation} \label{f: ellB}
\norm{\ell^B}{\ldO{B}}\leq
\mu(B)^{1/2}\norm{\ell}{H^{1,2}({\mu})^*}
\end{equation}

For every $f$ in $BMO^{r'}({\mu})$ and every finite linear combination
$g$ of $(1,r)$-atoms the integral $\int_{\BR^d} f\, g \wrt{\mu }$ is 
convergent.
Let $H_{\mathrm{fin}}^{1,r}({\mu})$ denote the subspace of $H^{1,r}({\mu})
$
consisting of all finite linear combinations of $(1,r)$-atoms.
Then $g \mapsto \int_{\BR^d} f\, g \wrt {\mu }$ defines a linear
functional on $H_{\mathrm{fin}}^{1,r}({\mu})$.  We observe that
$H_{\mathrm{fin}}^{1,r}({\mu})$ is dense in $H^{1,r}({\mu})$.

\begin{theorem} \label{t: duality}
Suppose that $r$ is in $(1,\infty)$.  The following hold
\begin{enumerate}
\item[\itemno1]
for every $f$ in $BMO^{r'}({\mu})$ the functional $\ell$, initially defined
on $H_{\mathrm{fin}}^{1,r}({\mu})$ by the rule
$$
\ell(g) = \int_{\BR^d} f\, g \wrt {\mu },
$$
extends to a bounded functional on $H^{1,r}({\mu})$.  Furthermore,
$$
\norm{\ell}{H^{1,r}({\mu})}
\leq \norm{f}{BMO^{r'}({\mu})};
$$
\item[\itemno2]
there exists a constant $C$ such that for every continuous linear
functional $\ell$ on $H^{1,r}({\mu})$ there exists a function
$f^\ell$ in $BMO^{r'}({\mu})$ such that
$\norm{f^\ell}{BMO^{r'}({\mu})} \leq C \,
\norm{\ell}{H^{1,r}({\mu})^*}$ and
$$
\ell(g) = \int_{\BR^d} f^\ell\, g \wrt {\mu }
\quant g \in H_{\mathrm{fin}}^{1,r}({\mu}).
$$
\end{enumerate}
\end{theorem}

\begin{proof}
The proof of \rmi\ follows the line of the proof of \cite{CW}
which is based on the classical result of C.~Fefferman \cite{F, FS}.
We omit the details.

Now we prove \rmii\ in the case where $r$ is equal to $2$.
The proof for $r$ in $(1,\infty)\setminus\{2\}$
is similar and is omitted.

Let $\ell$ be a bounded linear functional on $H^{1,2}({\mu})$. Fix  $\nu\in
\BZ$ and $b\in \BR^+$ as in Lemma  \ref{l: geometric}, such that $b$ is 
also greater than $R_0/(1-\beta)$, where $R_0$ and $\beta$ are the 
constants of assumption (AM). Recall that for all $b'\ge b$ the space
$H^{1,2}(\mu)$ is isomorphic to $H_{b'}^{1,2}(\mu)$ with norm
$\norm{\cdot}{H_{b'}^{1,2}(\mu)}$, by Proposition \ref{p: decphi} . Thus, we 
may interpret $\ell$ as a
continuous linear functional on $H_{b'}^{1,2}(\mu)$ for all $b'\geq
b$. Fix a  point $o$ in $\fZ^\nu$. For each $b'\ge b$ there exists
a function $\ell^{B(o,b')}$ in $L^2_{0}(B(o,b'))$ that represents $\ell$ as 
functional on 
$L^2\big(B(o,b')\big)$. Since both $\ell^{B(o,b)}$ and the restriction of $
\ell^{B(o,b')}$
to $B(0,b)$ represent the same functional on $L^2_0\big(B(o,b)\big)$, 
there exists a constant $\eta^{B(0,b')}$ such
that
$$
\ell^{B(o,b)}-\ell^{B(o,b')} = \eta^{B(0,b')}
$$
on $B(o,b)$. By integrating both sides of this equality on $B(o,b)$
we see that
$$
\eta^{B(0,b')}= -\frac{1}{\mu\bigl(B(o,b)\bigr)} \int_{B(o,b)} \ell^{B(o,b)}
\wrt \mu.
$$
Note that, since $\ell^{B(o,b)}\in L^2_0\big(B(o,b)\big)$,
\begin{equation}\label{media0}
\eta^{B(o,b)}=0.
\end{equation}
Define
$$
f^{\ell} (x) = \ell^{B(0,b')}(x) + \eta^{B(0,b')} \quant x \in B(o,b') \quant b'\ge 
b.
$$
It is straightforward to check that this is a good definition.

We claim that the function $f^\ell$ is
in $BMO({\mu})$ and there exists a constant~$C$ such that
$$
\norm{f^{\ell}}{BMO({\mu})}
\leq C \, \norm{\ell}{H^{1,2}({\mu})^*}
\quant \ell \in H^{1,2}({\mu})^*.
$$
First we show that $N_{b}^2(f^\ell) \leq
\norm{\ell}{H^{1,2}({\mu})^*}$. Indeed, choose a ball $B$ in
$\cB_{b}$.
 Then there exists a function $\ell^B$  in $\ldO{B}$ that represents the
restriction of $\ell$ to $\ldO{B}$ and a constant $\eta^B$ such that
\begin{equation} \label{f: fell on B}
f^\ell \big\vert_B = \ell^B + \eta^B.
\end{equation}
By integrating both sides on $B$, we see that $\eta^B = \bigl(f^\ell)_B$.
Thus, by (\ref{f: fell on B}) and (\ref{f: ellB}),
$$
\begin{aligned}
\Bigl( \frac{1}{\mu (B)} \int_B \bigmod{f^\ell - \bigl(f^\ell\bigr)_B}^2
\wrt {\mu } \Bigr)^{1/2}
& = \Bigl( \frac{1}{\mu (B)} \int_B \bigmod{\ell^B}^2
      \wrt {\mu } \Bigr)^{1/2} \\
& \leq \norm{\ell}{H^{1,2}({\mu})^*},
\end{aligned}
$$
so that $N_{b}^2(f^\ell) \leq \norm{\ell}{H^{1,2}({\mu})^*}$, as
required.\par
 Next we show that $f^\ell$ is in $\lu\mu$ and that
$\norm{f^\ell}{1} \leq C \, \norm{\ell}{H^{1,2}(\mu)^*}$. Let
 $\set{B_\alpha}$ be the covering described in Lemma~\ref{l: geometric}.
For each integer  $h\geq
2$ let $A_h$ denote the annulus $B(o,hb) \setminus B \bigl(o,(h-1)b
\bigr)$. For the sake of brevity denote $B(o,b)$  by
$B^o$. Observe that $M=B^o\bigcup\left(\bigcup_{h=2}^\infty A_h\right)$.  
The left inequality in Lemma~\ref{l: geometric} \rmi\
implies that
\begin{equation} \label{f: main estimate}
\begin{aligned}
\bignorm{f^\ell}{1} & = \bignorm{f^\ell}{\lu{B^o}}
        + \sum_{h=2}^\infty \bignorm{f^\ell}{\lu{A_h}} \\
& \leq \bignorm{\ell^{B^o}}{\lu{B^o}}
        + \sum_{h=2}^\infty \sum_{\{B_\alpha:
          B_{\al} \cap A_h \neq \emptyset\}}
          \bignorm{f^\ell}{\lu{B_{\al}}}.
\end{aligned}
\end{equation}
By (\ref{f: fell on B}), the triangle inequality,
the Schwarz inequality and (\ref{f: ellB})
\begin{equation} \label{f: estimate on balls}
\begin{aligned}
\bignorm{f^\ell}{\lu{B_{\al}}} &  \leq \mu(B_{\al})^{1/2} \,
\bignorm{\ell^{B_{\al}}}{\ldO{B_{\al}}}
      + \mu(B_{\al}) \, \mod{\eta^{B_{\al}}} \\
&  \leq \mu(B_{\al}) \, \norm{\ell}{H^{1,2}(\mu)^*}
      + \mu(B_{\al}) \, \mod{\eta^{B_{\al}}}.
\end{aligned}
\end{equation}

Now, we claim that if $B_\al \cap A_h \neq \emptyset$, then
\begin{equation}  \label{f: ball and its mother}
\bigmod{\eta^{B_\al}}
\leq8\Bigl(\frac{2d}{b}\Bigr)^{1/[1-\log_2(1+\be)]}
\, \sqrt{D} \,
        \norm{\ell}{H^{1,2}(\mu)^*},
        \end{equation}
where $D=D_{b/(a_0\de^\nu),a_0\, \de^\nu}$ is the doubling constant 
corresponding to the parameters $b/(a_0\de^\nu)$ and $a_0\, \de^\nu$ 
(see Remark \ref{r: geom I}), and $d$ denotes the distance  of $o$  from 
the centre $z_{\al}^\nu$ of $B_\alpha$.

By Lemma~\ref{l: geometric}~\rmii\ there exists a chain of points
$z_{\al_1}^\nu, \ldots, z_{\al_N}^\nu$, such that
$o=z_{\al_1}^\nu$, $z_{\al}^\nu =  z_{\al_N}^\nu$, with
$$
N \leq 4\, \Bigl(\frac{2d}{b}\Bigr)^{1/[1-\log_2(1+\be)]} + 1,
$$
and such that for all $j$ in $\{1,\ldots, N-1\}$ the intersection
$B_{\al_j} \cap B_{\al_{j-1}}$ contains the ball
$B(z^\nu_{\al_{j+1}},a_0\, \de^\nu)$.
Denote by $B_{\al_{j}}'$ the ball $B(z_{\al_j}^\nu,a_0\, \de^\nu)$.
Since, by (\ref{f: fell on B}), $\ell^{B_{\al_{j-1}}} + \eta^{B_{\al_{j-1}}}
= \ell^{B_{\al_{j}}} + \eta^{B_{\al_{j}}}$
on $B_{\al_{j-1}} \cap B_{\al_{j}}$, hence on $B_{\al_{j}}'$,
$$
\begin{aligned}
\bigmod{\eta^{B_{\al_j}}}
& \leq \bigmod{\bigl(\ell^{B_{\al_{j-1}}}
         + \eta^{B_{\al_{j-1}}}\bigr)_{{B_{\al_j}'}}}
         + \bigmod{\bigl(\ell^{B_{\al_j}}\bigr)_{{B_{\al_j}'}}} \\
& \leq \Bigl(\frac{1}{\mu({B_{\al_j}'})}
        \int_{{B_{\al_j}'}} \bigmod{\ell^{B_{\al_{j-1}}}}^2
        \wrt \mu \Bigr)^{1/2}
+ \bigmod{\eta^{B_{\al_{j-1}}}}
        + \Bigl(\frac{1}{\mu({B_{\al_j}'})}
        \int_{B_{\al_j}'} \bigmod{\ell^{{B_{\al_j}}}}^2 \wrt \mu \Bigr)^{1/2}
\end{aligned}
$$
by the triangle inequality and Schwarz's inequality.
Now we use (\ref{f: ellB}) to estimate the first and the
third summand and obtain that
\begin{equation}
\begin{aligned}
\bigmod{\eta^{B_{\al_j}}}
& \leq \sqrt{\frac{\mu\bigl(B_{\al_{j-1}}\bigr)}{\mu({B_{\al_j}'})}}
        \, \norm{\ell}{H^{1,2}(\mu)^*} +  \bigmod{\eta^{B_{\al_{j-1}}}}
        +  \sqrt{\frac{\mu({B_{\al_j}})}{\mu({B_{\al_j}'})}}
        \, \norm{\ell}{H^{1,2}(\mu)^*} \\
& \leq 2\, \sqrt{D} \,
        \norm{\ell}{H^{1,2}(\mu)^*}  + \bigmod{\eta^{B_{\al_{j-1}}}}.
\end{aligned}
\end{equation}
 Note that we have used (\ref{f: ratio of balls})
in Lemma~\ref{l: geometric} \rmii\ in the last inequality.
Hence, iterating this inequality, we obtain
\begin{align*}
\mod{\eta^{B_\alpha}}=\mod{\eta^{B_{\alpha_N}}}&\le 2(N-1)\sqrt{D}\,
\norm{\ell}{H^{1,2}(\mu)^*}+\mod{\eta^{B_0}} \\ 
&\le 8\Bigl(\frac{2d}{b}\Bigr)^{1/[1-\log_2(1+\be)]}\ \sqrt{D}\,\norm{\ell}{H^
{1,2}(\mu)^*},
\end{align*} 
because $\eta^{B_0}=0$. This proves the claim (\ref{f: ball and its 
mother}).\par
Now (\ref{f: estimate on balls}) and (\ref{f: ball and its mother})
imply that for all the balls of the covering $\set{B_\alpha}$
\begin{equation} \label{f: estimate fell}
\bignorm{f^\ell}{\lu{B_{\al}}} \leq \Bigl[1+ 8\Bigl(\frac{2d}{b}\Bigr)^{1/[1-
\log_2(1+\be)]}  \,
\sqrt{D} \, \Bigr] \, \mu(B_\alpha)
\, \norm{\ell}{H^{1,2}(\mu)^*},
\end{equation}
where $d$ denotes the distance $\rho(z_{\al}^\nu,o)$.
Note that if $B_\alpha\cap A_h\not=\emptyset$ then $d \leq (h+1) \, b$.
\par 
We estimate the first summand in (\ref{f: main estimate}) by
Schwarz's inequality and (\ref{f: ellB}), while we use
(\ref{f: estimate fell}) to estimate the other summands, and obtain that
$$
\bignorm{f^\ell}{1}
\leq \norm{\ell}{H^{1,2}(\mu)^*}\left(\mu(B^o) \, 
        + C\, 
        \sum_{h=2}^\infty (h+1)^{1/[1-\log_2(1+\be)]}
        \sum_{\{B_{\al} :
        B_{\al} \cap A_h \neq \emptyset\}}  \,
        \mu\bigl(B_{\al}\bigr)\right).
$$
Since the balls $\{B_{\al_j}\}$ have the finite
intersection property by Lemma~\ref{l: geometric}~\rmi\
and each such ball intersects at most three annuli $A_h$,
we have that
$$
\bignorm{f^\ell}{1}
\leq \mu(B^o) \, \norm{\ell}{H^{1,2}(\mu)^*}
   + C\, \norm{\ell}{H^{1,2}(\mu)^*}
    \sum_{h=2}^\infty (h+1)^{1/[1-\log_2(1+\be)]}
    \, \sum_{j = h-2}^{h+2} \mu\bigl(A_j\bigr).
$$
By Proposition~\ref{p: PPc}~\rmii\ there exist constants $\eta$ in
$(0,1)$  and $C>0$ such that $\mu(A_j) \leq C\, \eta^j$. Thus
$$
\sum_{h=2}^\infty (h+1)^{1/[1-\log_2(1+\be)]}
    \, \sum_{j = h-2}^{h+2} \mu\bigl(A_j\bigr)
< \infty,
$$
and we may conclude that
$$
\bignorm{f^\ell}{1}
\leq C \, \norm{\ell}{H^{1,2}(\mu)^*},
$$
thereby proving that $f^\ell$ is in $\lu\mu$.
\end{proof}

\begin{remark}
Note that the proof of Theorem~\ref{t: duality} does not
apply, strictly speaking, to the case where $r$ is equal to
$\infty$.  However, a straightforward, though tedious, adaptation to the 
case
where $\mu$ is only locally doubling
of a classical result \cite{CW}, show that $H^{1,\infty}(\mu)$
and $H^{1,2}(\mu)$ agree, with equivalence of norms.
Consequently, the dual space of $H^{1,\infty}(\mu)$ is $BMO(\mu)$.
\end{remark}

\section{Interpolation} \label{s: interpolation}

In this section we prove, for the finite measure case, the analogues of the 
interpolation theorems proved in \cite{CMM} when $\mu(M)=\infty$. 
Because of the close similarity with the infinite measure case, we shall be 
rather sketchy in our exposition and we shall only indicate the necessary 
modifications to the statements and the proofs.\par
The first technical ingredient in the proof of the interpolation theorems in 
\cite{CMM} is a covering lemma (see \cite[Prop. 5.3]{CMM}). To prove the 
analogous result for spaces that satisfy the complementary isoperimetric 
property we need a lemma. We recall that $B_0$ is the ball in the 
complementary isoperimetric property {\cIP} (see  Section \ref{s: 
PPc}). \par
\begin{lemma} \label{l: precovering II}
Suppose that $A$ is a open subset of $M$ such that 
$A\cap \Bar B_0$ is contained in $A_t$ for some $t$ in $\BR^+$.
Then
$$
\mu(A_t) \geq \bigl(1-\e^{-{\IcM} t/2}\bigr) \, \mu(A)/2. 
$$
\end{lemma}
\begin{proof}
First we prove that $\bigl(A\cap \Bar B_0^c\bigr)_t$  
is contained in $A_{2t}$.

Indeed, suppose that $x$ is in $\bigl(A\cap \Bar B_0^c\bigr)_t$.
Then either $x$ is in $A_t$, hence in $A_{2t}$, or
$x$ is in $\bigl(A\cap \Bar B_0^c\bigr)_t \setminus A_t$.
In the latter case 
$x$ is in $A\cap \Bar B_0^c$, and $\rho(x,A^c) > t$.
Furthermore $\rho(x,B_0) \leq t$, for otherwise the ball
$B(x,t)$ would be contained in $A\cap \Bar B_0^c$, 
i.e., $\rho(x,A^c\cup B_0)>t$,
contradicting the fact that $x$ is in $\bigl(A\cap \Bar B_0^c\bigr)_t$.  

Therefore the ball $B(x,t)$ is contained in $A$ and there
exists a point $y$ in $A\cap \Bar B_0$ such that $\rho(x,y) < t$. 
By assumption $y$ is in $A_t$, whence
$$
\rho(x,A^c) < \rho (x,y) + \rho (y,A^c) < 2t,
$$
as required. 

Now,
$$
\begin{aligned}
\mu(A)
& =    \mu (A\cap \Bar B_0) + \mu (A\cap \Bar B_0^c)  \\
& \leq \mu (A_t) + \bigl(1-\e^{-{\IcM} t}\bigr)^{-1} \, 
     \mu \bigl((A\cap \Bar B_0^c)_t\bigr)  \\
& \leq \mu (A_{2t}) + \bigl(1-\e^{-{\IcM} t}\bigr)^{-1} \, 
     \mu (A_{2t})  \\
& =  \frac{2-\e^{-{\IcM} t}}{1-\e^{-{\IcM} t}} \, 
     \mu (A_{2t}), 
\end{aligned}
$$
fro{}m which the desired estimate follows directly.
\end{proof}
 \begin{lemma} \label{l: covering II}
 Suppose that $\nu$ is an integer. For every $\kappa$  in $\BR^+$, every  
open subset $A$  of $M$ such
that $A\cap \Bar B_0 \subseteq A_{\kappa}$ and every collection $\cC$ is 
of dyadic cubes of resolution at least $\nu$
such that $\,\bigcup_{Q\in \cC} Q =A$, there exist mutually
disjoint cubes $Q_1,\ldots,Q_k$ in $\cC$ such that
\begin{enumerate}
\item[\itemno1]
$\sum_{j=1}^k \mu (Q_j) \geq \bigl(1-\e^{-{\IcM}\, \kappa/2}\bigr) \,
\mu (A)/4$;
\item[\itemno2]
$ \rho (Q_j, A^c) \leq \kappa$
for every $j$ in $\{1,\ldots, k\}$.
\end{enumerate}
\end{lemma}
\begin{proof}
 
The proof is almost \emph{verbatim} the same as the proof of
\cite[Proposition~3.5]{CMM}. The only difference is that we use Lemma 
\ref{l: precovering II} in the proof of \rmi.
\end{proof}

\begin{remark}
Observe that in Remark~\ref{r: Ic} we may substitute $B_0$ with any
ball containing $B_0$. Therefore we may assume that $r_{B_0} \geq
C_1\, \de^2$.
\end{remark}

The second technical ingredient is a relative distributional inequality for the 
\emph{noncentred dyadic maximal function}
\begin{equation} \label{f: HL}
\cM_2 f(x)
= \sup_{Q} \frac{1}{\mu (Q)} \int_Q \mod{f} \wrt \mu
\quant x \in M,
\end{equation}
where the supremum is taken over all dyadic cubes of
resolution $\geq 2$ that contain $x$, and the local sharp function
$$
f^{\sharp,b}(x)
= \sup_{B\in\cB_b (x)} \frac{1}{\mu (B)} \int_B \mod{f-f_B} \wrt \mu
\quant x \in M.
$$

Observe that $f$ is in $BMO(\mu)$ if and only if $f\in L^1(\mu)$ and
$\norm{f^{\sharp,b}}{\infty}$ is finite for some (hence for all)
$b>R_0/(1-\beta)$. \par
Note that the maximal operator $\cM_2$ is of weak type $1$. We denote 
by $\opnorm{\cM_2}{1;1,\infty}$ its weak type $1$ quasi norm. \par
For every $\alpha>0$  denote by $A(\alpha)$ and $S(\alpha)$
the level sets $\{\cM_{2}f>\alpha\}$ and $\{f^{\sharp,b'}>\alpha\}$
respectively. Thus, for $\alpha$ and $\epsilon>0$ 
$$
\set{\cM_2f>\alpha,\ f^{\sharp,b'}\le\epsilon\alpha}=A(\alpha)\cap S
(\epsilon\alpha)^c.
$$\par 
The following lemma is the analogue of \cite[Lemma 7.2]{CMM} for  
spaces of finite measure that satisfy the complementary isoperimetric 
property.
\begin{lemma} \label{l: rdi II}
Let $B_0$ be as in Remark~\ref{r: Ic}, with $r_{B_0}\geq C_1\,
\de^2$. Define constants $b'$, $\si$ and $D$ by
$$
b' = 2 C_1+C_0, \qquad \si = \bigl(1-\e^{-{\IcM}\,
C_1\de^2/2}\bigr)/4 \qquad\hbox{and}\qquad D =D_{b'/a_0,a_0},
$$
where $a_0$, $C_1$ and $\de$ are as in Theorem~\ref{t: dyadic
cubes}, and $D_{b'/a_0,a_0}$ is defined in Remark~\ref{r: geom I}. 
Denote by $\om$ the number
$$
\inf \Bigl\{\mu(Q): Q \in \cQ^2, Q\cap \Bar B_0 \neq \emptyset \Bigr\},
$$
and by $\fM$ a constant $> \opnorm{\cM_2}{1;1,\infty}/\om$.
Then for every $\eta'$ in $(0,1)$, for all  positive $\vep<(1-\eta')/(2D)$,
and for every $f$ in $\lu{\mu}$
$$
\mu \bigl(A({\al})\cap S({\vep\al})^c \bigr) \leq \eta\, \mu
\bigl(A({\eta'\al}) \bigr) \quant \al\geq \frac{\fM}{\eta'}\,
\norm{f}{\lu{\mu}}
$$
where
\begin{equation}\label{eta}
\eta=
1-\si + \smallfrac{2\vep \,D}{\si \,(1-\eta')}.
\end{equation}
\end{lemma}

\begin{proof}
First we prove that $\om$ is strictly positive. Indeed, suppose that
$Q_{\al}^2$ is a dyadic cube of resolution $2$ with nonempty
intersection with~$B_0$; the cube $Q_{\al}^2$ contains the ball
$B(z_{\al}^2,a_0 \, \de^2)$ by Theorem~\ref{t: dyadic cubes}~\rmv\
and is contained in the ball $2B_0$ by the triangle inequality.

Denote by $D$ the doubling constant $D_{a_0\de^2/(2r_{B_0}),a_0\de^2}
$.
By the local doubling property
$$
\begin{aligned}
\mu(2B_0)
& \leq D \, \mu\bigl(B(z_{\al}^2,a_0\, \de^2)\bigr) \\
& \leq D \, \mu\bigl(Q_{\al}^2\bigr).
\end{aligned}
$$
Therefore
$
\om \geq D^{-1} \, \mu(2B_0)>0,
$
as required.\par
For the rest of this proof we shall write $\kappa$ instead of $C_1\,
\de^2$.
Suppose that $\al \geq \fM \, \norm{f}{\lu{\mu}}/\eta'$. Since $f$ is
in $\lu{\mu}$,
\begin{equation} \label{f: level set}
\begin{aligned}
\mu(A(\eta'\al))
& \leq \frac{\opnorm{\cM_2}{1;1,\infty}}{\eta'\, \al} \, \norm{f}{\lu{\mu}} \\
& \leq \frac{\opnorm{\cM_2}{1;1,\infty}}{\fM}   \\
& <    \om.
\end{aligned}
\end{equation}
We claim that $\bigl(A(\eta'\al)\bigr)^{\kappa}=\set{x\in A(\eta'\al): \rho(x,A
(\eta'\al)^c)>\kappa}$ is contained in
$\Bar B_0^c$. Indeed, if $x$ is in $\bigl(A(\eta'\al)\bigr)^\kappa$,
and $Q$ is the dyadic cube of resolution $2$ that contains $x$, then
$Q$ is contained in $A(\eta'\al)$ by the triangle inequality.
Therefore $\mu(Q) \leq \mu \bigl(A(\eta'\al)\bigr) < \om$ by
(\ref{f: level set}). Hence $x$ is not in $\Bar B_0$ by the
definition of $\om$. The claim proved above implies that
$A(\eta'\al)\cap \Bar B_0\subseteq \bigl(
A(\eta'\al)\bigr)_{\kappa}$.

The rest of the proof is  the same as that of
\cite[Lemma 7.2]{CMM}. The only difference is that we use Lemma \ref{l: 
covering II} instead of \cite[Prop. 5.3]{CMM}.
\end{proof}

Next,  we prove the analogue of \cite[Theorem~7.3]{CMM}.

\begin{theorem} \label{t: basic II}
For each $p$ is in $(1,\infty)$
there exists a positive constant $C$ such that
$$
\norm{f}{\lu{\mu}} + \norm{f^{\sharp,b'}}{\lp{\mu}}
\geq C\, \norm{f}{\lp{\mu}}
\quant f \in \lp{\mu}.
$$
\end{theorem}

\begin{proof}
Observe that it suffices to show that
\begin{equation} \label{f: sharp max II}
\norm{f}{\lu{\mu}} + \norm{f^{\sharp,b'}}{\lp{\mu}}
\geq C\, \norm{\cM_{2}f}{\lp{\mu}},
\end{equation}
because $\cM_{2}f \geq \mod{f}$ by the differentiation theorem
of the integral.

Let $\sigma$ and  $\fM$ be as in the statement of Lemma~\ref{l: rdi II}. Fix 
$\eta'=(1-\sigma/4)^{1/p}$ and let $\eta$ be as in (\ref{eta}).
Denote by $\xi$ the number $\fM\, \norm{f}{\lu{\mu}}/\eta'$.
Then
$$
\begin{aligned}
\quad\norm{\cM_{2}f}{p}{^{p}} 
 &= p \ioty \al^{p-1}\, \mu \bigl(A(\al) \bigr) \wrt \al  \\
& = p \int_{\xi}^\infty \al^{p-1}
        \, \bigl[\mu \bigl(A(\al) \cap S(\vep\al)^c\bigr)
      +  \mu \bigl(A(\al) \cap S(\vep\al)\bigr) \bigr] \wrt \al  \\
&\hskip3truecm      + p \int_0^\xi \al^{p-1} \, \mu\bigl(A(\al)\bigr) \wrt \al,
\end{aligned}
$$
so that, by Lemma~\ref{l: rdi II},
$$
\begin{aligned}
\norm{\cM_{2}f}{p}{^{p}} 
& \leq  p\, \eta \ioty \al^{p-1}
        \, \mu \bigl(A(\eta'\al)\bigr) \wrt \al
      + p \ioty \al^{p-1}
        \, \mu \bigl(S(\vep\al)\bigr) \wrt \al\\
&\hskip3truecm      + p\,\mu(M) \int_0^\xi \al^{p-1}  \wrt \al \\
& =  p\, \eta\, \eta'^{-p} \ioty \gamma^{p-1}
        \, \mu \bigl(A(\gamma)\bigr) \wrt \gamma
      + p\,  \vep^{-p} \ioty \gamma^{p-1}
        \, \mu \bigl(S(\gamma)\bigr) \wrt \gamma
      + \mu(M)\,\xi^p    \\
& \leq  \eta\, \eta'^{-p} \, \norm{\cM_{2}f}{p}{^{p}}  + \vep^{-p} \,
        \norm{f^{\sharp,b'}}{p}^p
       +\mu(M) \frac{\fM^p}{(\eta')^p} \norm{\cM_{2}f}{1}{^{p}} .
\end{aligned}
$$
Now we choose $\vep$ small enough so that $\eta \leq 1-\si/2$.
Therefore $\eta\, \eta'^{-p}<1$ and (\ref{f: sharp max II}) follows.
\end{proof}
If $X$ and $Y$ are Banach spaces and  $\te$ is in
$(0,1)$, we denote by $(X,Y)_{[\te]}$ the complex interpolation
space between $X$ and $Y$ with parameter $\te$.\par
Now that all the groundwork has been laid, we may proceed to state the 
interpolation theorems without further ado. The proofs are adaptations of classical results. We refer the reader to \cite[Theor. 7.4 and Theor. 7.5]
{CMM} for more details.
\begin{theorem} \label{t: interpolation}
Suppose that $\te$ is in $(0,1)$.  The following hold:
\begin{enumerate}
\item[\itemno1]
if $p_\te$ is $2/(1-\te)$, then
$\bigl(\ld{\mu},BMO(\mu)\bigr)_{[\te]} = L^{p_\te}(\mu)$;
\item[\itemno2]
if $p_\te$ is $2/(2-\te)$, then
$\bigl(\hu{\mu},\ld{\mu}\bigr)_{[\te]} = L^{p_\te}(\mu)$.
\end{enumerate}
\end{theorem}
\begin{theorem}
Let $S$ denote the strip $\{z\in \BC: \Re z\in (0,1)\}$.
Suppose that $\{\cT_z\}_{z\in \bar S}$ is a family of uniformly bounded
operators on $\ld{\mu}$ such that $z\mapsto \int_{\BR^d} T_zf \, g \wrt \mu 
$
is holomorphic in $S$ and continuous in $\bar S$ for all
functions $f$ and $g$ in $\ld{\mu}$.
Further, assume that there exists a constant $A$ such that
$$
\opnorm{T_{is}}{\ld{\mu}} \leq A
\qquad\hbox{and}\qquad
\opnorm{T_{1+is}}{\ly{\mu};BMO({\mu})} \leq A.
$$
Then for every $\te$ in $(0,1)$ the operator $T_\te$ is bounded
on $L^{p_\te}({\mu})$, where $p_\te = 2/(1-\te)$, and
$$
\opnorm{T_\te}{L^{p_\te}({\mu})}
\leq A_\te,
$$
where $A_\te$ depends only on $A$ and on $\te$.
\end{theorem}

\section{Singular integrals}\label{singint}
In this section we state the analogue of Theorem 8.2 in \cite[Theor. 8.2]
{CMM}. 
Assume that $\cT$ is a bounded linear operator on $\ld{\mu}$ with kernel 
$k$; i.e. $k$ is a function on $M\times M$ which is locally integrable  off 
the diagonal in $M\times M$ and such
that for every function $f$ with support of finite measure
$$
\cT f(x) = \int_M k(x,y) \, f(y) \wrt \mu(y) \quant x \notin \supp
f.
$$

\begin{theorem} \label{t: singular integrals}
Suppose that $b$ is in $\BR^+$ and $b>R_0/(1-\be)$, where $R_0$ and
$\be$ appear in the definition of property \EB. Suppose that $\cT$
is a bounded operator on $\ld{\mu}$ and that its kernel~$k$ is
locally integrable off the diagonal of $M\times M$.  Let
$\upsilon_{k}$ and $\nu_{k}$ be defined by
$$
\upsilon_{k} = \sup_{B\in \cB_{b}} \sup_{x,x'\in B} \int_{(2B)^c}
\mod{k(x,y) - k(x',y)} \wrt \mu (y),
$$
and
$$
\nu_{k} = \sup_{B\in \cB_{b}} \sup_{y,y'\in B} \int_{(2B)^c}
\mod{k(x,y) - k(x,y')} \wrt {\mu }(x).
$$
The following hold:
\begin{enumerate}
\item[\itemno1]
if $\nu_{k}$ is finite, then $\cT$ extends to a bounded operator on
$\lp{\mu}$ for all $p$ in $(1,2]$ and from $H^1({\mu})$ to
$\lu{\mu}$. Furthermore, there exists a constant $C$ such that
$$
\opnorm{\cT}{H^1({\mu});\lu{\mu}} \leq C \, \bigl(\nu_{k} +
\opnorm{\cT}{\ld{\mu}}\bigr);
$$
\item[\itemno2]
if $\upsilon_{k}$ is finite, then $\cT$ extends to a bounded
operator on $\lp{\mu}$ for all $p$ in $[2,\infty)$ and from
$\ly{\mu}$ to $BMO({\mu})$. Furthermore, there exists a constant $C$
such that
$$
\opnorm{\cT}{\ly{\mu};BMO({\mu})} \leq C \, \bigl(\upsilon_{k} +
\opnorm{\cT}{\ld{\mu}}\bigr);
$$
\item[\itemno3]
if $\cT$ is self adjoint on $\ld{\mu}$ and $\nu_{k}$ is finite, then
$\cT$ extends to a bounded operator on $\lp{\mu}$ for all $p$ in
$(1,\infty)$, from $H^1({\mu})$ to $\lu{\mu}$ and from $\ly{\mu}$ to
$BMO({\mu})$.
\end{enumerate}
\end{theorem}

\begin{proof}
The proof is \emph{almost verbatim} the same as the proof of
\cite[Thm~8.2]{CMM}, and is omitted.
\end{proof}

\begin{remark} \label{r: singular integrals}
It is worth observing that in the case where $M$ is a Riemannian
manifold and the kernel $k$ is ``regular'', then the condition
$\upsilon_{k} <\infty$ of Theorem~\ref{t: singular integrals}~\rmi\
may be replaced by the condition $\upsilon_{k}'<\infty$, where
\begin{equation} \label{f: singular integrals}
\upsilon'_{k} = \sup_{B\in \cB_{b}} \, r_B \sup_{x\in B}
\int_{(2B)^c} \mod{\nabla_x k(x,y)} \wrt \mu (y).
\end{equation}
Similarly, the condition $\nu_{k} <\infty$ of Theorem~\ref{t:
singular integrals}~\rmii\ may be replaced by the condition
$\nu_{k}'<\infty$, where
\begin{equation} \label{f: singular integrals II}
\nu'_{m} = \sup_{B\in \cB_{b}} \, r_B \sup_{y\in B} \int_{(2B)^c}
\mod{\nabla_y k(x,y)} \wrt {\mu }(x).
\end{equation}
\end{remark}

\section{Riemannian manifolds}\label{s: Cheeger}

Let $(M,\rho,\mu)$ be a complete Riemannian manifold of dimension $d$, 
endowed with the Riemannian metric $\rho$ and the corresponding 
Riemannian measure $\mu$. Let $h(M)$ be Cheeger's isoperimetric 
costant, defined by
$$
h(M)= \inf  \frac{\sigma(\partial A)}{\mu(A)}
$$
where the infimum runs over all bounded open sets $A$  with smooth 
boundary  $\partial A$ such that $\mu(A)\le \mu(M)/2$.   Here $\sigma$ 
denotes the induced $(d-1)$-dimensional Riemannian measure on $
\partial A$. Note that the  condition  $\mu(A)\le \mu(M)/2$ is automatically 
satisfied if $\mu(M)=\infty$.\par
In \cite[Section 9]{CMM} we proved that, on Riemannian manifolds of 
infinite measure,  the isoperimetric property (I) is  equivalent to the 
positivity of $h(M)$.
 Moreover, if the Ricci curvature is bounded from below, both properties 
are equivalent to the positivity of the bottom of the spectrum of $M$
$$
b(M)=\inf\set{\int_M\mod{\nabla f}^2\wrt\mu: f\in C^1_c(M), \norm{f}{2}=1}.
$$\par
Here we shall prove that, when $M$ has finite measure, an analogous 
characterization holds for the complementary isoperimetric property {\cIP}, provided that  we replace  $b(M)$ by the spectral gap of the 
Laplacian
$$
\lambda_1(M)=\set{\int_M\mod{\nabla f}^2\wrt\mu: f\in C^1_c(M), \norm{f}
{2}=1\ \textnormal{and}\ \int_M f\wrt\mu=0}.
$$
\par
Again, since the arguments coincide to a large extent with those used to 
prove \cite[Theor. 9.5]{CMM}, we point out only the differences, referring 
the reader to \cite{CMM} for details 
and unexplained terminology. \par 
Given a measurable set $E$ in $M$, we shall denote by $P(E)$ its 
perimeter, i.e. the total variation $\Var(\1_E,M)$ in $M$ of the indicator 
function $\1_E$ of $E$.
The following lemma is the counterpart of \cite[Prop. 9.2]{CMM}, in the 
finite measure case.
\begin{lemma}\label{l: cheeger for bd}
Suppose that $M$ is a complete unbounded Riemannian manifold of
finite volume. If $h(M)>0$, then for every measurable set $E$ with
$\mu(E)\leq \mu(M)/2$
$$
P(E)\ge h(M)\, \mu(E).
$$
\end{lemma}
\begin{proof} Let $f$ be a real-valued function in $C_c^1(M)$, whose support has measure less than $\mu(M)/2$. By the coarea formula
\cite{Chavel}, 
$$
\int_M\mod{\nabla f}\wrt\mu\ge h(M)\, \int_M\mod{f}\wrt\mu.
$$ 
By \cite[Prop. 1.4]{MPPP},  there exists
a sequence $(f_n)$ of functions in $C^1_c(M)$, whose support has measure less than $\mu(M)/2$, such that $f_n\to
\1_E$ in $L^1(M)$ and \break $\int_M\mod{\nabla f_n}\wrt\mu\to\Var(\1_E,M)=P
(E)$.  Hence, passing to the limit, we get
$
P(E)\ge\,h(M)\, \mu(E).
$
\end{proof}
Now we are ready to state the main result of this section. 
We recall that the constant ${\IcM}$ is defined
in (\ref{IcM}).

\begin{theorem} \label{t: Cheeger and (P)}
Suppose that $M$ is a complete unbounded Riemannian manifold of
finite volume and  Ricci curvature bounded from below.  Then the
following are equivalent:
\begin{enumerate}
\item[\itemno1] $h(M)>0$;
\item[\itemno2]
$M$
possesses property {\cIP}\ ;
\item[\itemno3]  $\lambda_1(M)>0$;
\end{enumerate}
\end{theorem}
\begin{proof}
To prove that \rmi\  implies \rmii \ , we fix a ball $B_0$ such that 
$\mu(B_0)>\mu(M)/2$ and we consider a open set $A$ in $M\setminus 
\Bar{B}_0$. Fix $t>0$ and let $f$ be the function defined by
$$
f(x)=\begin{cases}
t^{-1}\, \rho(x,A^c) &\textnormal{if}\  x\in A_t  \\
1 & \textnormal{if}\ x\in A\setminus A_t \\
0 & \textnormal{if} \ x\in A^c.
\end{cases}
$$
Then $f$ is Lipschitz and $\mod{\nabla f(x)}=t^{-1}$ for almost every $x$ 
in $A_t$, $\mod{\nabla f(x)}=0$ elsewhere. Thus, by the coarea formula 
for  functions of bounded variation \cite{EG, M} and Lemma \ref{l: cheeger for bd},
\begin{align*}
t^{-1}\,\mu(A_t)=\int\mod{\nabla f}\wrt\mu=
&\int_0^1 P\big(\set{f=s}\big)\wrt s \\
\ge&\,h(M)\int_0^1\mu\big(\set{f>s}\big) \wrt s \\ 
=&\,h(M) \int f\wrt \mu \\
\ge&\,h(M) \big(\mu(A)-\mu(A_t)\big). 
\end{align*}Thus
$$
\mu(A_t)\ge \frac{h(M)}{1+h(M)t}\,t\,\mu(A) \qquad\forall t>0.
$$
Hence $M$ satisfies property {\cIP} and  by (\ref{IcM}) 
the constant $\IcM$ is at least $h(M)$.\par
Next, we prove that \rmii\ implies \rmiii.
Let $A$ be a bounded open set with regular boundary, contained in 
$M\setminus \Bar{B}_0$. Then  $\mu(A_t)\geq
(1-e^{-I_{M,\Bar B_0}t})\mu(A)$ for all $t>0$, by Proposition \ref{p: PPc}.
 Since the boundary of $A$ is regular,
 $$
\sigma(\partial A)=\lim_{t\to 0+}\frac{ \mu(A_t)}{t}\ge \,I_{M,B_0}\,\mu(A).
$$
Hence, by  the coarea formula, for every real-valued function $f$ 
in $C^\infty_c(M\setminus \Bar{B}_0)$
\begin{equation}\label{eq: local1}
    I_{M,B_0}\int_M\mod{f}\wrt\mu\leq \int_M\mod{\nabla f}\wrt\mu.
    \end{equation}
    By
replacing $f$ with $f^2$ in (\ref{eq: local1}), 
we obtain that
\begin{equation}\label{eq: local2}
\inf\frac{\int_M\mod{\nabla f}^2\wrt\mu}{\int_M\mod{f}^2\wrt\mu}\geq
\frac{I^2_{M, B_0}}{4},
\end{equation}
where the infimum is taken over  all real $f$ in
$C^\infty_c(M\setminus \Bar B_0)$, such that $\norm{f}{2}\not=0$. Hence  the  bottom 
of the essential spectrum $b_{\textrm{ess}}(M)$ of the Laplace--Beltrami operator on $M$  
is positive, by the variational characterization of $b_{\textrm{ess}}(M)$ \cite{Br}. 
Thus $0$ is an isolated point in the spectrum and
$\lambda_1(M)>0$.
\par
Finally, to prove that \rmiii\ implies \rmi, we use the fact that, if the Ricci 
curvature is bounded below by $-K$ for some $K
\ge 0$, then
$$
\lambda_1(M)\le C\big(\sqrt{K}h(M)+h(M)^2\big),
$$ 
where $C$ is a constant which depends only on the dimension of $M$ 
\cite{Bu,Le}.
\end{proof}

\section{Another family of metric spaces}
\label{s: Another}

In this section we shall construct another family of metric measured
spaces which are locally doubling and satisfy the approximate midpoint property 
and  the isoperimetric property. They may have either infinite or finite measure. 
In the first case they satisfy property (I), in the latter case property {\cIP}\ (see Remark \ref{r: compare}  
or \cite{CMM} for the definition of property (I)). 
The
spaces we consider are of the form $(\BR^d,
\rho_\varphi,\mu_{\varphi})$ or $(\BR^d,
\rho_\varphi,\mu_{-\varphi})$, where $\varphi$ is a function in
$C^2(\BR^d)$ which satisfies certain additional conditions specified
later, $\rho_\varphi$ is the Riemannian metric on $\BR^d$ defined by
the length element $\wrt s^2=(1+\mod{\nabla \varphi})^2\,(\wrt
x_1^2+\cdots+\wrt x_d^2)$ and
$\wrt\mu_{\pm\varphi}=\e^{\pm\varphi}\wrt\lambda$. Note that $\mu_{\pm\varphi}$ is not the Riemannian metric on $(\BR^d,\rho_\varphi)$. First we need some
preliminaries on Riemannian metrics of the form $\wrt s^2=
m^2\,(\wrt x_1^2+\cdots+\wrt x_d^2)$, where $m$ is a continuous
positive function on $\BR^d$ which tends to infinity at infinity.
\par We say that a positive function $m\in C^0(\BR^d)$ is
\emph{tame} if for every $R>0$ there exists a constant $C(R)\ge1$
such that
$$
C(R)^{-1}\le \frac{m(x)}{m(y)}\le C(R) \qquad\forall x,y \in\BR^d {\rm\ \, 
such\ that\ }\, \mod{x-y}<R.
$$
 \par
The following lemma provides a simple criterion for establishing tameness.
\begin{lemma}\label{AC}
Let $m$ be a function in $C^1(\BR^d)$ such that $m\ge 1$
and $\mod{\nabla m}\,\le C\, m^\alpha$
 for some $\alpha$ in $[0,1]$ and some $C>0$. Then $m$ is tame.
\end{lemma}
 \begin{proof}
 By the mean value theorem, for all $x,y$ in $\BR^d$ such that $\mod{x-y}<R$,
 \begin{align*}
\left|\log\frac{m(x)}{m(y)}\right|&\le\,\mod{x-y}\,\max_{z\in\BR^d} 
\frac{\mod{\nabla m(z)}}{m(z)}\,\le\,CR. 
\end{align*}
\end{proof}
It is easy to see that the functions $m(x)=1+\mod{x}^\alpha$, with $
\alpha\ge 0$, are tame. The function $\e^{\mod{x}^\alpha}$ is tame
 if and only if $0\le \alpha\le 1$.\par
\begin{proposition}\label{equivalence}
Let $m$ be a tame function such that $\lim_{x\to\infty}m(x)=\infty$.  
Denote by $\rho$ the Riemannian metric on $\BR^d$ defined by
 the length element $\wrt s^2=m(x)^2\,(\wrt x_1^2+\cdots+\wrt x_d^2)$. 
Then the manifold $(\BR^d,\rho)$ is complete.
  Moreover, for every $R>0$, there exists a constant $C(R)\ge1$ such that 
for all $x$, $y$ in $\BR^d$ with $\rho(x,y)<R$
\begin{equation}\label{eq}
C(R)^{-1}\,m(x)\,\mod{x-y}\,\le\,
\rho(x,y)\,\le\,C(R)\,m(x)\,\mod{x-y}.
\end{equation}
\end{proposition}
\begin{proof} The function $m$ has a positive minimum on $\BR^d$, which 
we may assume to be greater
 than or equal to one, by multiplying $m$ by a positive constant if 
necessary. If $\gamma$ is a path
  in $\BR^d$ we shall denote by $\ell(\gamma)$ its length with respect to the 
Riemannian metric $\rho$ and by
   $\ell_e(\gamma)$ its Euclidean length. Since the minimum of $m$ on $
\BR^d$ is at least $1$ we
    have that $\ell(\gamma)\ge\,\,\ell_e(\gamma)$ for all paths $\gamma$. 
Hence
\begin{equation}\label{r>e}
\rho(x,y)\ge\,\mod{x-y} \qquad\forall x,y \in \BR^d.
\end{equation}
Let $x$ and $y$ be two points in $\BR^d$ such that $\rho(x,y)<R$ and 
denote by $\gamma$ be the
 segment of line joining them. Since $\mod{x-y}\le\rho(x,y)< R$ and $m$ is tame,
\begin{equation}\label{second}
\rho(x,y)\le \ell(\gamma)=\int_0^1 m\big(\gamma(t)\big)\,\mod{\gamma'(t)}
\wrt t\le \, C(R) \, m(x)\,\mod{x-y}.
\end{equation}
This proves the second inequality in (\ref{eq}).\par
Together the two inequalities (\ref{r>e}) and (\ref{second})
imply that the manifold $(\BR^d,\rho)$ is complete.  In particular any two 
points in $(\BR^d,\rho)$ may be joined by a minimizing geodesic by the 
Hopf-Rinow theorem.\par It remains to prove the first inequality  in (\ref
{eq}).  We observe that
there exists a constant $A$ such that for all $S>0$ there exists a  compact 
set $K(S)$ in $\BR^d$ such that
\begin{equation}\label{unif}
A^{-1}\le \frac{m(x)}{m(y)}\le A \qquad\forall x,y \in\BR^d \ \,{\rm such \ that\ 
\,}  x\notin K(S),  \ \mod{x-y}<S/m(x).
\end{equation}
Indeed, by the definition of tame function it suffices to choose $A=C(1)$ and $K(S)=\set{x
\in \BR^d: m(x)\le S}$.\par
Fix $R>0$ and let $x,y$ be in $\BR^d$ such that $\rho(x,y)<R$.
Assume first that $x\notin K(AR)$ and let $\gamma:[0,\rho(x,y)]\to\BR^d$  
be a minimizing geodesic joining $x$ and $y$. We claim that $\mod
{\gamma(t)-x}<AR/m(x)$ for all $t$ in $[0,\rho(x,y)]$.  Indeed, suppose by 
contradiction that   there exists  $t_0$ in $\big[0,\rho(x,y)\big]$ such that $
\mod{\gamma(t_0)-x}=AR/m(x)$ and  $\mod{\gamma(t)-x}<AR/m(x)$ for 
all $t$ in $[0,t_0)$. Then, by (\ref{unif})
 \begin{align*}
\rho(x,y)\ge& \int_0^{t_0} m\big(\gamma(t)\big)\,\mod{\gamma'(t)} \wrt t \\
\ge& \,A^{-1} \,m(x)\,\mod{\gamma(t_0)-x} \\
= &R, \end{align*}
which contradicts the assumption $\rho(x,y)<R$. Thus the claim is proved.  
Hence by (\ref{unif})
\begin{align*}
\rho(x,y)&=\int_0^{\rho (x,y)}m\big(\gamma(t)\big)\,\mod{\gamma'(t)} \wrt t 
\nonumber\\
&\ge A^{-1}\,m(x)\,\mod{y-x}.
\end{align*}
Finally, if $x\in K(AR)$ by (\ref{r>e})
$$
m(x)\,\mod{x-y}\le\, m(x)\ \rho(x,y) \le\, \max_{K(AR)}m\ \rho(x,y).
$$
This concludes the proof of the proposition.
\end{proof}
\noindent
\begin{proposition}\label{p: ld}
Let $\varphi$ be a function in $C^1(\BR^d)$ such that $\lim_{x\to\infty}
\mod{\nabla\varphi(x)}=\infty$ and $1+\mod{\nabla\varphi}$ is tame. Then 
the metric measure spaces $(\BR^d, \rho_\varphi,\mu_\varphi)$ and $(\BR^d, \rho_\varphi,\mu_{-\varphi})$ are locally 
doubling.
\end{proposition}
\begin{proof} Write $m(x)=1+\mod{\nabla\varphi(x)}$ for the sake of 
brevity.
Let $B_e(x,r)$ denote the Euclidean ball of centre $x$ and radius $r$ in $
\BR^d$. We claim that for every $R>0$ there exists a constant $D(R)$ 
such that
\begin{equation}\label{D(R)}
D(R)^{-1}\,\e^{\varphi(x)}\,\le\,\e^{\varphi(y)}\,\le\,D(R)\,\e^{\varphi(x)}
\qquad\forall y\in B_e\big(x,R/m(x)\big).
\end{equation}
Indeed, by the mean value theorem and the fact that $m$ is tame
\begin{align*}
\mod{\varphi(x)-\varphi(y)}&\le \max\set{ \mod{\nabla \varphi(z)}\ \mod{x-
y}: z\in B_e\big(x,R/m(x)\big)}\\
&\le \,C(R)\, m(x) \,\mod{x-y}\\
&\le\, C(R)\,R,
\end{align*}
whence (\ref{D(R)}) follows with $D(R)=\e^{C(R)\,R}$. Thus for every 
$R>0$
\begin{equation}\label{eqmeas}
D(R)^{-1}\,\e^{\varphi(x)}\ \le \,\frac{\mu_\varphi\Big(B_e\big(x, r/m(x)\big)
\Big)}{\lambda\Big(B_e\big(x, r/m(x)\big)\Big)}\le\,D(R)\,\e^{\varphi(x)} 
\qquad\forall x\in\BRd\quad 0<r\le R.
\end{equation}
Thus  $(\BR^d, \rho_\varphi,\mu_\varphi)$ is locally doubling, because by Proposition \ref{equivalence} there 
exists a constant $C$ (which depends on $R$ but not on $r$) such that
$$
B_e\big(x,C^{-1}\,r/m(x)\big)\subset B(x,r)\subset B(x,2r)\subset B_e\big(x,
2C\,r/m(x)\big)\qquad\forall r\in[0,R]
$$
and the Lebesgue measure is doubling. The proof for $(\BR^d, \rho_\varphi,\mu_{-\varphi})$ is similar.
\end{proof}
Next, we look for sufficient conditions that guarantee that the
spaces $(\BR^d,\rho_\varphi, \mu_\varphi)$ and $(\BR^d, \rho_\varphi,\mu_{-\varphi})$ satisfy the isoperimetric property.
\begin{definition}\label{admissible}
Let $\varphi$ be function in  $C^1(\BR^d)$. We say that $\varphi$ is 
\emph{admissible} if
\begin{itemize}
\item[\rmi] there exists $\tau_0>0$ such that $\varphi$ is  $C^2$ for $\mod
{x}\ge \tau_0$;
\item[\rmii] ${1+\mod{\nabla\varphi}}$ is tame and
$$\lim_{x\to\infty}\mod{\nabla\varphi(x)}=\infty,\quad\lim_{x\to\infty}\frac
{\mod{{\mathrm {Hess}}\,\varphi(x)}}{\mod{\nabla\varphi(x)}^2}=0;
$$
\item[\rmiii]  the radial derivative $\partial_r\varphi=\smallfrac{x}{\mod{x}}
\cdot\nabla\varphi$  satisfies
$$
\liminf_{x\to\infty}\frac{\partial_r\varphi(x)}{\mod{\nabla\varphi(x)}}>0.
$$
\end{itemize}
\end{definition}
It is easy to see that the functions $\mod{x}^\alpha$, with $\alpha> 1$ are 
admissible. The function $\e^{\mod{x}^\alpha}$ is not admissible if $\alpha>1$.
\begin{lemma}\label{Hopital}
Let $\psi:[0,\infty)\to\BR$ be a continuous function such that $\psi\in C^2\big([\tau_0,
\infty)\big)$ for some $\tau_0>0$. Assume that
$$
\liminf_{r\to\infty} \psi'(r)>0,\quad \lim_{r\to\infty}\frac{\psi''(r)}{\big(\psi'(r)
\big)^2}=0.
$$
Let $h$ be  a positive function in $C^0\big([0,\infty)\big)$ such that
$$
\liminf_{r\to\infty} h(r)\,\psi'(r)>0.
$$
Then for every $d\ge1$ there exists  a positive constant $C$ such that $$
\int_\tau^{\tau+ a\,h(\tau)} \e^{\psi(r)}r^{d-1}\wrt r\ge \,C\, a\,\int_0^{\tau+a
\,h(\tau)} \e^{\psi(r)}r^{d-1}\wrt r\qquad\forall \tau\in\BR_+\quad\forall a\in 
[0,1].
$$
\end{lemma}
\begin{proof}
It is clearly enough to prove that
$$
\int_\tau^{\tau+ a\,h(\tau)} \e^{\psi(r)}r^{d-1}\wrt r\ge \,C\, a\,\int_0^{\tau} \e^
{\psi(r)}r^{d-1}\wrt r\qquad\forall \tau\in\BR_+.
$$
The integral in the right hand side is asymptotic to
$\e^{\psi(\tau)}\tau^{d-1}/\psi'(\tau)$ as $\tau$ tends to infinity,
by l'H\^opital's rule and the assumptions on $\psi$. Let
$\tau_1>\tau_0$ be such that
\begin{equation}\label{star1}
\int_0^\tau\e^{\psi(r)}\,r^{d-1}\wrt r\le 2\,\e^{\psi(\tau)}\,\frac{\tau^{d-1}}
{\psi'(\tau)}\qquad\forall \tau\ge\tau_1.
\end{equation}
The assumptions on $\psi$ and $h$ imply that if we choose $\tau_1$ 
sufficiently large there exists $\eta>0$ such that
$$
\psi'(\tau)\ge\eta,\quad h(\tau)\,\psi'(\tau)\ge\,\eta \qquad\forall \tau\ge
\tau_1.
$$
Thus, if $\tau>\tau_1$ the function $\psi$ is increasing. Hence for $\tau>
\tau_1$
\begin{align*}
\int_\tau^{\tau+ a\,h(\tau)}\e^{\psi(r)}r^{d-1}\wrt r&\ge\,\e^{\psi(\tau)}\,\tau^
{d-1}\, a\,h(\tau)  \\
&\ge \,{\eta}\,{a} \,\e^{\psi(\tau)}\,\frac{\tau^{d-1}}{\psi'(\tau)} \\
&\ge\,\frac{\eta}{2}\,{a}\,\int_0^\tau\e^{\psi(r)}\,r^{d-1}\wrt r,
\end{align*}
where in the last inequality we have used (\ref{star1}). 
It remains to prove the desired inequality for $\tau$ in $[0,\tau_1]$. Set
$
m_0=\min_{[0,\infty]}\psi$, $M_0=\max_{[0,\tau_1]}\psi$ and $h_0=\min_
{[0,\tau_1]}h$. Then for $\tau\in[0,\tau_1]$
$$
\int_0^\tau \e^{\psi(r)}\,r^{d-1}\wrt r\le\,\e^{M_0}\ \frac{\tau^d}{d}
$$
and
\begin{align*}
\int_\tau^{\tau+ a\,h(\tau)}\e^{\psi(r)}\,r^{d-1}\wrt r&\ge\,\e^{m_0}\,\tau^{d-1}
\, a\,h(\tau)\\
&\ge\, \e^{m_0}\,  \tau^{d}\,  a\,h_0/\tau_1.
\end{align*}
This implies that the desired inequality holds also for $\tau$ in $[0,\tau_1]
$.
\end{proof}
\begin{lemma}\label{Hopital^c}
Let $\psi$ and $h$ two functions which satisfy the assumptions of Lemma 
\ref{Hopital}. Assume further that
$$
\lim_{r\to\infty}\big(r-h(r)\big)=\infty.
$$
Then for every $d\ge1$ there exist  positive constants $C$ and $T$ such 
that
$$
\int_{\tau-a\,h(\tau)}^\tau \e^{-\psi(r)}r^{d-1}\wrt r\ge \,C\, a\,\int_{\tau- a\,h
(\tau)}^\infty \e^{-\psi(r)}r^{d-1}\wrt r\qquad\forall \tau\ge\,T \quad\forall a\in 
[0,1].
$$
\end{lemma}
\begin{proof}
It is clearly enough to prove that
$$
\int_{\tau-a\,h(\tau)}^\tau \e^{-\psi(r)}r^{d-1}\wrt r\ge \,C\, a\,\int_{\tau}^\infty 
\e^{-\psi(r)}r^{d-1}\wrt r\qquad\forall
 \tau\ge\,T \quad\forall a\in [0,1].
$$
The integral in the right hand side is asymptotic to
$\e^{-\psi(\tau)}\tau^{d-1}/\psi'(\tau)$ as $\tau$ tends to
infinity, by l'H\^opital's rule and the assumptions on $\psi$. Thus
there exists $\tau_1>\tau_0$  such that
\begin{equation}\label{star2}
\quad\int_\tau^\infty\e^{-\psi(r)}\,r^{d-1}\wrt r\le
2\,\e^{-\psi(\tau)}\,\frac{\tau^{d-1}}{\psi'(\tau)}\qquad\forall
\tau\ge\tau_1.
\end{equation}
The assumptions on $\psi$ and $h$ imply that if we choose $\tau_1$ 
sufficiently large there exists $\eta>0$ such that
$$
\psi'(\tau)\ge\eta,\quad h(\tau)\,\psi'(\tau)\ge\,\eta, \quad r\psi'(r)>d-1
\qquad\forall\tau\ge\tau_1.
$$
Note that the last inequality implies that the function $r\mapsto \e^{-\psi(r)}
\, r^{d-1}$ is decreasing for $r>\tau_1$. Choose $T>\tau_1$ such that $
\tau-h(\tau)>\tau_1$ for $\tau\ge T$. Then for $\tau\ge T$
\begin{align*}
\int_{\tau-a\,h(\tau)}^\tau\e^{-\psi(r)}r^{d-1}\wrt r&\ge\,\e^{-\psi(\tau)}\,\tau^
{d-1}\, a\,h(\tau)  \\
&\ge \,{\eta}\,{a} \,\e^{-\psi(\tau)}\,\frac{\tau^{d-1}}{\psi'(\tau)} \\
&\ge\,\frac{\eta}{2}\,{a}\,\int_\tau^\infty\e^{-\psi(r)}\,r^{d-1}\wrt r
\end{align*}
where in the last inequality we have used (\ref{star2}). 
This concludes the proof of the lemma.
\end{proof}

\begin{theorem}\label{adm>P}
Suppose that the function $\varphi$ is admissible. Then
\begin{itemize}
\item[\rmi] the measured
metric space $(\BR^d, \rho_\varphi,\mu_\varphi)$ is locally
doubling, $\mu_\varphi(\BR^d)=\infty$, and satisfies property \PP;
\item[\rmii] the space $(\BR^d, \rho_\varphi,\mu_{-\varphi})$ is locally
doubling, $\mu_{-\varphi}(\BR^d)<\infty$, and satisfies property
{\cIP}, for some ball $B_0\subset \BR^d$.
\end{itemize} 
\end{theorem}
\begin{proof}
Both spaces are locally doubling by Proposition \ref{p: ld}. It easily
follows from the assumptions on $\varphi$ that
$\mu_\varphi(\BR^d)=\infty$ and $\mu_{-\varphi}(\BR^d)<\infty$. To prove
that $(\BR^d, \rho_\varphi,\mu_\varphi)$ satisfies also property \PP\ 
we must prove that there exists a constant $C$ such that for every
bounded open set $A$ and every $\kappa$ in $[0,1)$
$$
\mu_\varphi(A_\kappa)\,\ge\,C\,\mu_\varphi(A),
$$
where we recall that $A_\kappa=\set{x\in A:\rho(x,A^c)<\kappa}$. \par
Henceforth we shall write $m=1+\mod{\nabla\varphi}$, for the sake of 
brevity.  Since $m$ is tame there exists a constant $C_1\ge1$ such that
$$
C_1^{-1}\le\, \frac{m(x)}{m(y)}\,\le\, C_1 \qquad\forall x,y \ \,{\rm such\ that\ 
\,} \mod{x-y}<1.
$$
Let $d$ denote the Euclidean distance in $\BR^d$ and set
$$
A'_\kappa=\set{x\in A: d(x,A^c)<\frac{\kappa}{C_1\,m(x)}}.
$$
We observe that if $x\in A'_\kappa$, then there exists $y$ in $A^c$
such that
$$
\mod{x-y} < \frac{\kappa}{C_1\,m(x)}\,\le\, 1
$$
Thus, by (\ref{eq}), we get that $\rho(x,y)<\,C_1\,m(x)
\,\mod{x-y}<\kappa$. Hence $A'_\kappa\subset A_\kappa$ and it
suffices  to prove that there exists a constant $C$ such that
$$
\mu_\varphi(A'_\kappa)\,\ge\,C\,\mu_\varphi(A).
$$
\par
For every $\om$ in the unit sphere $S^{d-1}$ let $\mu_\varphi^\om$ 
denote the measure on $\BR_+$ defined by
$$
\mu_\varphi^\om\big(E\big)=\int_E \e^{\varphi(r\om)}\,r^{d-1} \wrt r
$$
for every measurable subset $E$ of $\BR_+$. \par The functions
$\psi_\om(r)=\varphi(r\om)$ and $h_\om(r)=1/m(r\om)$ satisfy the
assumptions of Lemma~\ref{Hopital} uniformly with respect to $\om$
in $S^{d-1}$. Thus for all $a\in [0,1]$ there exists a constant
$C>0$ such that
\begin{equation}\label{muof}
\mu_\varphi^\om\big((\tau,\tau+a \,h_\om(\tau)\big)\ge\,C\,a
 \,\mu_\varphi^\om\big([0,\tau+a \,h_\om(\tau))\big)\qquad\forall \tau\in\BR
+, \quad \forall \om\in S^{d-1}.
\end{equation}
If $F$ is a measurable subset of $\BR^d$ let $F(\om)$ denote the set $
\set{r\in\BR_+: r\om\in F}$.  \par
If the set $(A\setminus A'_\kappa)(\om)$ is empty then obviously
$$
\mu_\varphi^\om\big(A'_\kappa(\om)\big)=\mu_\varphi^\om\big(A(\om)
\big).
$$
Otherwise, set $\tau_\om=\sup(A\setminus A'_\kappa)(\om)$. 
Observe that $\tau_\omega\omega\in A\setminus A'_\kappa$. 
Indeed, by the definition of $\tau_\omega$,  there exists a sequence 
$s_n\to \tau_\omega$ such that $s_n\omega\in A\setminus A'_\kappa$. 
By the continuity of $m$
$$
d(\tau_\omega\omega, A^c)=\lim_nd(s_n\omega, A^c)\ge\lim_n 
\frac{\kappa}{C_1\,m(s_n\omega)} =\frac{\kappa}{C_1\,m(\tau_\omega\omega)}>0.
$$
This implies that $\tau_\omega\omega\in A\setminus A'_\kappa$.\par 
The set 
$(A\setminus A'_\kappa)(\om)$ is obviously contained
 in the interval $[0,\tau_\om)$.  We claim that  the set $A'_\kappa(\om)$ contains the 
interval $\big(\tau_\om,\tau_\om+C_1^{-1}\,\kappa\,h_\om(\tau_\om)\big)$.
 Indeed, if $s\in\big(\tau_\om,\tau_\om+C_1^{-1}\,\kappa\,h_\om(\tau_\om)\big)$, 
 then $d(\tau_\omega\omega,s\omega)<\kappa/\big(C_1\,m(\tau_\omega\omega)\big)$. 
 Hence $s\omega\in A$, because otherwise $\tau_\omega\omega$ would be in $A'_\kappa$. 
 Since $s\omega\notin A\setminus A'_\kappa$ by the definition of $\tau_\omega$, 
 the claim is proved.\par Then, writing $a=C_1^{-1}\,\kappa$ for the sake 
 of brevity, using the fact 
that
  for every positive number $\delta$ the function $x\mapsto x/(\delta+x)$ is 
increasing and (\ref{muof}), we see that
\begin{align*}
\frac{\mu_\varphi^\om\big(A'_\kappa(\om)\big)}{\mu_\varphi^\om\big(A
(\om)\big)}&=\frac{\mu_\varphi^\om\big(A'_\kappa(\om)\big)}{\mu_\varphi^
\om\big((A\setminus A'_\kappa)(\om)\big)+\mu_\varphi^\om\big(A'_\kappa
(\om)\big)} \\
&\ge\,\frac{\mu_\varphi^\om\big((\tau_\om,\tau_\om+a\,h_\om(\tau_\om)
\big)}{\mu_\varphi^\om\big([0,\tau_\om)\big)+\mu_\varphi^\om\big((\tau_
\om,\tau_\om+a\,h_\om(\tau_\om)\big)} \\
&=\,\frac{\mu_\varphi^\om\big((\tau_\om,\tau_\om+a\,h_\om(\tau_\om)
\big)}{\mu_\varphi^\om\big([0,\tau_\om+a\,h_\om(\tau_\om)\big)}\\
&\ge\, C\,a=CC_1^{-1}\kappa.
\end{align*}
Thus, integrating in polar coordinates, one has
\begin{align*}
\mu_\varphi(A'_\kappa)& =\int_{S^{d-1}} \mu_\varphi^\om\big(A'_\kappa
(\om)\big)\wrt\sigma(\om)
\\ &\ge \,C\,\kappa\,\int_{S^{d-1}} \mu_\varphi^\om\big(A(\om)\big)\wrt
\sigma(\om)\\
&=\,C\,\kappa\,\mu_\varphi(A).
\end{align*}
This concludes the proof of property \PP\ for $(\BR^d,
\rho_\varphi,\mu_{\varphi})$.\par

The proof of property {\cIP}\ for
$(\BR^d,
\rho_{-\varphi},\mu_{-\varphi})$ is similar. The main
differences are the following
\begin{itemize}
\item[\rmi] the set $A$ is a open set contained in the complement 
of $\set{x\in\BR^d: \mod{x}\ge\,T}$ for some $T>0$ which depends only 
on $\varphi$;
\item[\rmii] the definition of $\tau_\om$ now is $\inf{(A\setminus A'_\kappa)
(\om)}$;
\item[\rmiii]  the set $(A\setminus A'_\kappa)(\om)$ is contained in the 
interval $(\tau_\om,\infty)$ and the set $A'_\kappa(\om)$ contains the 
interval $\big(\tau_\om-C_1^{-1}\,\kappa\,h_\om(\tau_\om), \tau_\om,\big)
$;
\item[\rmiv] the use of Lemma \ref{Hopital^c} instead of Lemma \ref
{Hopital}.
\end{itemize}
We omit the details.
\end{proof}
\begin{remark}\label{}
We point out that the $H^1-BMO$ theory for the Gaussian space $(\BR^d,\gamma)$ developed in \cite{MM} is a particular case of the theory exposed in the present paper. Indeed, $\gamma=\mu_{-\varphi}$ with $\varphi(x)=\mod{x}^2$. Moreover, in \cite{MM} the family of admissible balls is the set  $\cB^\gamma_1$ of all Euclidean balls $B$ in $\BR^d$
such that 
$
r_B \leq \min\bigl(1,1/\mod{c_B}\bigr),
$
where $c_B$ and $r_B$ denote the centre and the radius of $B$
respectively, while the family $\cB_1$ of admissible balls in $(\BR^d,\rho_\varphi,\gamma)$ is the set of all balls of radius at most one, with respect to the metric 
$\wrt s^2=(1+\mod{x})^2(\wrt x_1^2+\cdots+\wrt x_d^2)$.
By Proposition \ref{equivalence}, every ball in $\cB^\gamma_1$ is contained in a ball in $\cB_1$ of comparable measure and viceversa. Thus the spaces $H^1(\gamma)$ and $BMO(\gamma)$ defined in \cite{MM} coincide with those defined in the present paper.
\end{remark}


\begin{thebibliography}{GMMST1}

\bibitem[Br]{Br} R. Brooks, On the spectrum of non-compact manifolds 
with finite volume, \emph{Math. Z.}, \textbf{187}
(1984), 425--432.

\bibitem[Bu]{Bu}
P. Buser, {A note on the isoperimetric constant}, \emph{Ann. Sci.
\'Ecole Norm. Sup.} \textbf{15} (1982), 213--230.

\bibitem[CMM]{CMM} A. Carbonaro, G. Mauceri and S. Meda,
\emph{$H^1$ and $BMO$ on certain measured metric spaces},
{\tt arXiv:0808.0146v1 [math.FA]}.

\bibitem[Cha]{Chavel} I. Chavel,
Isoperimetric inequalities, vol. 145 of Cambridge Tracts in
Mathematics. Cambridge University Press, Cambridge, 2001. Differential
geometric and analytic perspectives.


\bibitem[Ch]{Ch} M. Christ, \emph{A $T(b)$ theorem with remarks on 
analytic capacity and the Cauchy integral}, Colloq.
Math. \textbf{50/51} (1990), 601Ð628.


\bibitem[CW]{CW}
R.R. Coifman and G. Weiss, \emph{Extensions of Hardy spaces and their
use in analysis},
{Bull. Amer. Math. Soc.} \textbf{83} (1977), 569--645.


\bibitem[CJ]{CJ} M. Cwickel and S. Janson,
Interpolation of analytic families
of operators, \emph{Studia Math.} \textbf{79} (1984), 61--71.


\bibitem[D]{Da} G. David, \emph{Morceaux de graphes Lipschitziens et Int
«egrales Singuli«eres sur une surface}, Revista
Mat. Iberoamericana \textbf{1} (1985), 1Ð56.

\bibitem[EG]{EG} L.C. Evans and R. Gariepy,  Measure theory and fine properties of functions,
CRC Press, Boca Raton, Ann Arbor, and London 1992.


\bibitem[F]{F} C. Fefferman,
\emph{Characterizations of bounded mean oscillation},
{Bull. Amer. Math. Soc.} \textbf{77} (1971), 587--588.

\bibitem[FS]{FS} C. Fefferman and E.M. Stein,
\emph{$H^p$ spaces of several variables},
{Acta Math.} \textbf{87} (1972), 137--193.



\bibitem[I]{I} A.D. Ionescu, \emph{Fourier integral operators
on noncompact symmetric spaces of real rank one},  J. Funct. Anal.
\textbf{174}  (2000),  no. 2, 274--300.

\bibitem[Le]{Le}
M. Ledoux, {A simple analytic proof of an inequality of P.~Buser},
\emph{Proc. Amer. Math. Soc.} \textbf{121} (1994), 951--959.



\bibitem[MMNO]{MMNO}
J. Mateu, P. Mattila, A. Nicolau and J. Orobitg, \emph{BMO
for nondoubling measures}
Duke Math. J. \textbf{102} (2000), 533--565.

\bibitem[MM]{MM}
G.~Mauceri and S. Meda, \emph{$ BMO$ and $ H^1$ for the Ornstein-Uhlenbeck operator}, 
J. Funct. Anal. \textbf{252} (2007), 278--313. 

\bibitem[M]{M} M. Miranda (Jr), \emph{Function of Bounded Variations on 'good' Metric spaces}, J. Math. Pures Appl. \textbf{82} (2003), 975--1004. 

\bibitem[MPPP]{MPPP} M. Miranda Jr, D. Pallara, F. Paronetto, M. 
Preunkert,\emph{Heat Semigroup and Functions of
 Bounded Variation on Riemannian Manifolds}, J. Reine Angew. Math. {\bf 
613} (2007), 99--119. 

\bibitem[NTV]{NTV}  Nazarov, Treil and Volberg,
\emph{The $Tb$-theorem on non-homogeneous spaces}, {Acta Math.}
\textbf{190} (2003), no. 2, 151--239.


\bibitem[To]{To} X. Tolsa,
\emph{BMO, $H\sp 1$, and Calder\'on-Zygmund operators
for non doubling measures}, {Math. Ann.}
\textbf{319} (2001), no. 1, 89--149.

\bibitem[V]{V} J. Verdera,
On the $T(1)$-theorem for the Cauchy integral,
\emph{Ark. Mat.}
\textbf{38} (2000), 183--199.
\end{thebibliography}
\end{document}